\documentclass[11pt]{article}
\usepackage{mathrsfs} 
\usepackage[margin=1.25in]{geometry}
\usepackage{ifthen}
\usepackage{graphicx}
\usepackage{authblk}
\usepackage{epsfig}
\usepackage{dsfont, amssymb}
\usepackage{amsfonts,dsfont,amssymb,amsthm,stmaryrd,bbm}
\usepackage[cmex10]{amsmath} 

\usepackage{hyperref,mathtools}
\hypersetup{colorlinks=true,pdftitle=""}

\newtheorem{conj}{Conjecture}[section]
\newtheorem{thm}[conj]{Theorem}

\newtheorem{remark}[conj]{Remark}
\newtheorem{lem}[conj]{Lemma}
\newtheorem{prop}[conj]{Proposition}
\newtheorem{coro}[conj]{Corollary}

\newtheorem{defn}[conj]{Definition}
\newtheorem{cor}[conj]{Corollary}
\newtheorem{ex}[conj]{Example}

\newcommand\independent{\protect\mathpalette{\protect\independent}{\perp}} 
\def\independent#1#2{\mathrel{\rlap{$#1#2$}\mkern2mu{#1#2}}}

\newcommand{\supp}{\mathrm{Supp}}
\newcommand{\Conv}{\mathrm{Conv}}

\newcommand{\R}{\mathbb{R}}

\renewcommand{\P}{\mathbb{P}}

\newcommand{\E}{\mathbb{E}}

\newcommand{\Z}{\mathcal{Z}}

\newcommand{\conv}{\mathrm{conv}}

\newcommand{\N}{\mathbb{N}}
\renewcommand{\Z}{\mathbb{Z}}

\def\phi{\varphi}

\def\bee{\begin{eqnarray*}}
\def\ene{\end{eqnarray*}}

\DeclarePairedDelimiter\floor{\lfloor}{\rfloor}

\def\lc{\left\lceil}   
\def\rc{\right\rceil}

\def\lf{\left\lfloor}   
\def\rf{\right\rfloor}

\allowdisplaybreaks

\title{{ Geometric and Functional Inequalities for Log-Concave Probability Sequences}}
\author[1]{Arnaud Marsiglietti}
\author[2]{James Melbourne}

\affil[1]{Department of Mathematics\\
University of Florida\\
Gainesville, FL 32611, USA\\
E-mail: a.marsiglietti@ufl.edu, Phone: 352-294-2310}

\affil[2]{Probabilidad y Estad\'isticas\\
Centro de Investigaci\'ones en Matem\'aticas\\
Guanajuato, GTO 36023, MX}
\date{}

\begin{document}

\maketitle

\begin{abstract}

We investigate geometric and functional inequalities for the class of log-concave probability sequences.  We prove dilation inequalities for log-concave probability measures on the integers. A functional analogue of this geometric inequality is derived, giving large and small deviation inequalities from a median, in terms of a modulus of regularity. Our methods are of independent interest, we find that log-affine sequences are the extreme points of the set of log-concave sequences belonging to a half-space slice of the simplex. We use this result as a tool to derive simple proofs of several convolution type inequalities for log-concave sequences, due to Walkup, Gurvits, and Klartag-Lehec.  Further applications of our results are used to produce a discrete version of the Pr\'ekopa-Leindler inequality.

\end{abstract}


\vskip5mm

{\bf Keywords:} log-concave sequence, Krein-Milman, localization lemma, four function theorem, concentration inequality

\section{Introduction}

A sequence of positive numbers $p = \{p_0, p_1, \dots, p_n \}$ is called log-concave when it satisfies
\begin{align}
    p_i^2 \geq p_{i-1}p_{i+1}
\end{align}
for $1 \leq i \leq n-1$.  Such sequences occur naturally in a multitude of contexts.   In Probability and Statistics log-concavity is of interest in its connection with notions of negative dependence \cite{joag1983negative, Pem00, borcea2009negative}.  In Information Theory entropy maximizers among log-concave random variables have been studied in \cite{Joh07, JKM13, MT20:isit}.  Important sequences in Combinatorics are log-concave (or conjectured to be log-concave)  see \cite{schoenberg1955zeros, stanley1981two, stanley2000positivity, savage2015, BH20, HMMS22} for some examples.  Many log-concave sequences are proven such by the following result that goes back to Newton. If $\{p_i\}_{i=0}^m$ is a positive sequence of numbers such that $P(x) = \sum_{i=0}^m {m \choose i} p_i x^i$ is a polynomial with real zeros, then the sequence $p_i$ is log-concave.  In fact, positive sequences that produce real rooted polynomials in the manner described is a strictly stronger condition than usual log-concavity.  Such sequences are referred to as P\'olya frequency sequences, or real-rooted and are log-concave with respect to a binomial reference measure as we will describe later in this article. See \cite{pitman1997probabilistic} for probabilistic implications of a sequence being real-rooted.

The Alexandrov-Fenchel inequality \cite[Theorem 7.3.1]{schneider2014convex} provides another interesting source of log-concave sequences.  It is essentially due to Minkowski that the volume of convex bodies is a homogeneous polynomial. More explicitly, for compact convex sets $K_1$ and $K_2$ in $\mathbb{R}^d$ and $t_1, t_2 \geq 0$, there exist coefficients $V_i(K_1, K_2)$ such that 
\begin{align*}
    | t_1 K_1 + t_2 K_2|_d = \sum_{i=0}^d  {d \choose i} V_i(K_1,K_2) t_1^{d-i} t_2^i,
\end{align*}
where
$$ A + B = \{a+b : a \in A, b \in B\} $$
denotes the Minkowski sum of subsets $A, B \subset \R^d$, and $| \cdot |_d$ denotes the usual $d$-dimensional Lebesgue measure.  The Alexandrov-Fenchel inequality implies that the ``mixed volumes'' $V_i(K_1,K_2)$ form a log-concave sequence.  We direct the reader to \cite{mccoy2014steiner, amelunxen2014living, goldstein2017gaussian} for investigations of mixed volumes, in particular ``intrinsic volumes'', with applications to learning theory.

Discrete log-concave random variables, those given by a log-concave probability mass function, are a convolution stable class containing many fundamental discrete distributions, such as Bernoulli, binomial, geometric, hypergeometric, and Poisson distributions.
For further background on log-concavity see the survey papers \cite{Sta89, Bre94, saumard2014log, Br15}.

In this article we will pursue  geometric and functional inequalities for the class of log-concave probability sequences.  In particular we establish dilation inequalities for discrete log-concave probability measures in the form of Nazarov, Sodin and Volberg \cite{NSV02} (see also \cite{bobkov2008sharp}, \cite{Fra09}). More precisely, we prove in Theorem \ref{thm: geometric dilation inequailty} that if $\mu$ is a log-concave probability measure and $A \subset K$, where $K \subset \Z$ is a (possibly infinite) interval, then for all $\delta \in (0,1)$,
    \begin{align}\label{dilat}
    \mu(A) \geq \mu^\delta(A_\delta) \mu^{1-\delta}(K),
    \end{align}
where $A_{\delta}$ is defined in Definition \ref{defn: delta interior}.  We derive a functional version of \eqref{dilat}, that for $f \colon K \to \mathbb{R}$ and $\varepsilon >0$
\begin{align}
        \mu (\{ |f| > \lambda \varepsilon \}) \geq \mu^\delta (\{ |f| \geq \lambda \}),
    \end{align}
    where $\delta = \delta_f(\varepsilon)$ denotes the ``modulus of regularity'' of $f$.  As a corollary we attain small and large deviation inequalities for functions of log-concave random variables
    (see Corollary \ref{deviations}).
    Moreover we demonstrate that these inequalities reduce to sharp inequalities in the special case $f(x) = x$.  That is for $X$ log-concave on $\{1,2, \dots \}$ with median $\mathrm{Med}(X)$ and $t > 1$,
    \begin{align*}
        \mathbb{P}( X > \mathrm{Med}(X) t) \leq e^{- \log(2) t/2}, \hspace{5mm} \mathbb{P}( X \leq \mathrm{Med}(X)/t ) \leq 1 - e^{- 2 \log (2) /t}.
    \end{align*} 
    
An important reduction in the proof of the dilation inequality is obtained through identifying the extreme points of a half-space slice of the log-concave sequences to be log-affine sequences.   This phenomena of constrained optimization for ``concave measures''  reducing to constrained optimization for ``affine measures'' is well known in the continuous setting.  It can be understood as a discrete analogue of the localization technique \cite{LS93, KLS95, FG04,FG06, Eld13, Sergey2015localization, Sergey2016ProbabilitySurveys, Klar17} utilized in Asymptotic Convex Geometry and Computer Science for proving isoperimetric and concentration type inequalities (see, e.g., \cite{LS93,Dyer-Frieze,KLS95,Gue99,NSV02,NSV03,Bob07:gafa,Fra09,BM11:aop}), improving the algorithmic complexity of computing the volume of convex bodies (see, e.g., \cite{LS93,KLS95,KLS97,CV18}), and in particular making progress towards the solution of the Kannan-Lov\'asz-Simonovits (KLS) conjecture (see \cite{LV17:1,LV17:2}).  However, in contrast to the continuous setting, this approach is general, and can be used with respect to an arbitrary reference measure, not just the counting measure as would be anticipated from the continuous theory. More precisely, we prove that
$$  \sup_{\P_X \in \mathcal{P}_h(\llbracket M,N \rrbracket)} \Phi(\P_X) $$
is attained at the distribution of a random variable $X$ whose probability mass function is log-affine. Here, $\P_X$ denotes the distribution associated with a random variable $X$, $\Phi$ is an arbitrary convex function on $\mathcal{P}_h(\llbracket M,N \rrbracket)$ the set of all discrete log-concave distributions $\P_X$ supported on $\{M, \dots, N\}$, $M,N \in \Z$, satisfying $\E[h(X)] \geq 0$ for an arbitrary function $h \colon \{M, \dots, N\} \to \R$. As mentioned above, a more general statement involving log-concavity with respect to an arbitrary reference measure is also available (see Corollary \ref{maximize}).

We will discuss several other applications of our methods. For example, we obtain a Four Function theorem, which asserts that the inequality
\begin{equation*}
\E[f_1(X)]^\alpha \E[f_2(X)]^\beta \leq  \E[f_3(X)]^\alpha \E[f_4(X)]^\beta
\end{equation*}
holds for all $X$ log-concave random variable with respect to a reference measure $\gamma$ if and only if it holds for all log-affine random variable with respect to $\gamma$, where $f_1,f_2, f_3,f_4$ are nonnegative functions and $\alpha, \beta >0$ (Theorem \ref{thm: four function theorem}). We provide a simple proof of Walkup's theorem on the stability of the convolution of ultra log-concave sequences (Corollary \ref{Walkup}), and other convolution type inequalities (Corollaries \ref{KL} and \ref{Gur}). We also establish a discrete Pr\'ekopa-Leindler inequality in Theorem \ref{thm: main}. If $f$ and $g$ are nonnegative unimodal functions on $\Z$ and $\mu$ is a discrete log-concave measure, then the following discrete Pr\'ekopa-Leindler inequality holds
\begin{align}\label{discretePL}
    \int f \square_t g(z) d\mu(z) \geq \left( \int f(z) d\mu(z) \right)^{1-t} \left( \int g(z) d \mu(z) \right)^{t},
\end{align}
where $f \square_t g(z) = \sup_{\{(x,y):|(1-t)x + t y - z| < 1 \}} f^{1-t}(x) g^t(y)$. We will demonstrate how this inequality complements recent progress on Pr\'ekopa-Liendler inequalities in the discrete setting (see \cite{KL19, HKS, GRST19, S20}). In particular \eqref{discretePL} applied to indicators will be crucial in establishing the dilation inequality \eqref{dilat}.

This article can also be viewed as part of the recent trend on the so-called ``discretization of convex geometry'' where one wants to translate results from convex geometry to the discrete setting. Recent developments include discrete analogue of the Brunn-Minkowski inequality (see, e.g., \cite{GG01, OV12, KL19, HKS, LNZ20, GRST19, S20}), discrete analogue of Koldobsky's slicing inequality (see \cite{AHZ17}), discrete analogue of Aleksandrov's theorem (see \cite{RYZ17}).

The paper is organised as follows. In Section 2, we review the background on generalized log-concave random variables, and give examples of such classes of variables from the literature.  We then characterize the extreme points of half slices of the space of log-concave sequences
and recall how the Krein-Milman theorem can be used to reduce constrained optimization problems in this context. All applications are given in Section 3. We start by recovering a ``four function theorem'', popular in the continuous setting \cite{KLS95}, and then give simple proofs of several convolution inequalities by Fekete-P\'olya, Keilson-Gerber, and Hoggar \cite{fekete1912problem, KG71, hoggar1974chromatic}, Walkup \cite{Wal76}, Klartag-Lehec \cite{KL19}, and Gurvits \cite{gurvits2009multivariate}.  We then prove a discrete Pr\'ekopa-Leindler inequality, and utilize it in the subsequent dilation inequality in Theorem \ref{thm: geometric dilation inequailty}, which is followed by its functional corollaries.  We close the paper deriving large and small deviation inequalities in terms of quantiles.

\section{Characterization of Extreme Points} \label{sec: Extreme point char}

Throughout, $\Z$ denotes the set of integers equipped with its usual Euclidean structure $|\cdot|$. For $a, b \in \mathbb{Z}$ such that $a \leq b$, let us denote $\llbracket a, b \rrbracket = \{ x \in \mathbb{Z}: a \leq x \leq b \}$, and for $n \in \N$, let us denote $\llbracket n \rrbracket = \llbracket 0,n \rrbracket$. We will also use $\llbracket a , b \llbracket$ to denote $\{ x \in \mathbb{Z}: a \leq x < b \}$, $\rrbracket a , b \rrbracket$ to denote $\{ x \in \mathbb{Z}: a < x \leq  b \}$, and so on.

\begin{defn}

    A function $f \colon \mathbb{Z} \to [0,\infty)$ is log-concave when it satisfies
    \begin{align} \label{eq: simple log-concave equation}
        f^2(n) \geq f(n-1) f(n+1)
    \end{align}
    for all $n \in \Z$ and for all $a \leq b$, $a,b \in \{ f > 0\}$ implies $\llbracket a, b \rrbracket \subseteq \{ f > 0 \}$.  
\end{defn}








\begin{defn}

    A function $f \colon \mathbb{Z} \to [0,\infty)$ is log-affine if it satisfies
    \begin{align} \label{eq: simple log-affine equation}
        f^2(n) = f(n-1) f(n+1)
    \end{align}
    when $n-1,n,n+1 \in \{f>0\}$ and has contiguous support.
    
\end{defn}

We now introduce the class of integer valued random variables that we will work with. First, let us recall that the probability mass function (p.m.f.) associated with an integer valued random variable $X$ is
$$ p(n) = \mathbb{P}(X = n), \quad n \in \Z. $$
For an integer valued measure $\gamma$ with mass function $q$, defined as $q(n) = \gamma(\{n\})$, $n \in \Z$, and a random variable $X$ with p.m.f. $p$, we will call p.m.f. of $X$ with respect to $\gamma$ the ratio
$$ f(n) = \frac{p(n)}{q(n)}, \quad n \in \Z, $$
which equals 0 when $q(n) = 0$ by convention.

\begin{defn}[Generalized log-concave random variables]

Let $\gamma$ be an integer valued measure with a contiguous support on $\mathbb{Z}$. A random variable $X$ on $\mathbb{Z}$ is log-concave with respect to $\gamma$ when its p.m.f. with respect to $\gamma$ is a log-concave function.

\end{defn}

\begin{ex}[log-concave random variables]

The class of discrete log-concave random variables corresponds to taking $\gamma$ to be the counting measure, that is, with mass function $q \equiv 1$. In particular, log-concave random variables are those with a log-concave p.m.f.

\end{ex}

Most fundamental discrete random variables fall into the class of log-concave random variables. For example, Bernoulli, binomial, geometric, hypergeometric, and Poisson random variables are all log-concave.

The following sub-class of discrete log-concave random variables can be seen as an analogue of the strongly log-concave random variables in the continuous setting (that is, log-concave with respect to a Gaussian).

\begin{ex}[Ultra-log-concave random variables \cite{Pem00}]

A random variable $X$ on $\N$ is ultra log-concave when its p.m.f. with respect to $\gamma$, the law of a Poisson distribution, is log-concave.

\end{ex}

Note that an ultra-log-concave random variable has a contiguous support and a probability mass function $p$ satisfying the following inequality
$$ p^2(n) \geq \frac{n+1}{n} p(n+1)p(n-1), \quad n \geq 1. $$

\begin{ex}[Ultra-log-concave random variables of order $m$ \cite{Pem00}]
A random variable $X$ on $\N$ is ultra log-concave of order $m$ when its p.m.f. with respect $\gamma$, the law of a Binomial distribution $B(m,1/2)$, is log-concave.  Stated quantitatively, this corresponds to $X$ supported on $\llbracket m \rrbracket$ and its mass function $p$ satisfies
\begin{align*}
    p^2(n) \geq \frac{(n+1)(m-n+1)}{n(m-n)} p(n+1) p(n-1).
\end{align*}
\end{ex}

Note that $\frac{(n+1)(m-n+1)}{n(m-n)}$ is decreasing in $m$, so that the class of ultra-log-concave variables of order $m$ is contained in the ultra-log-concave variables of order $m'$, for $m' \geq m$.  Taking the limit $m \to \infty$ we obtain the ultra-log-concave variables.  As mentioned in the introduction, it is a classical result going back to Newton (see \cite{Sta89} for a proof), that if $b_i$ denote the coefficients of a degree $m$ polynomial $P(x)$ with real zeros, then the sequence $b_i$ is ultra log-concave of order $m$.

\begin{ex}[$q$-factor log-concavity \cite{mcnamara2010infinite}]
A random variable $X$ on $\N$ is $q$-factor log-concave (or $q$-weighted log-concave \cite{wang2008q}) for $q >0$ when its p.m.f. with respect to the measure $\gamma(\{n\}) = q^{-n^2/2}$ is log-concave.  This is equivalent to the statement that on its contiguous support the mass function $p$ satisfies
\begin{align*}
    p^2(n) \geq q p(n+1)p(n-1).
\end{align*}
\end{ex}


%

We next describe the class of log-affine random variables.

\begin{defn}[Generalized log-affine random variables]

Let $\gamma$ be an integer valued measure with a contiguous support on $\mathbb{Z}$ and mass function $q$. A random variable $X$ on $\mathbb{Z}$ with p.m.f. $p$ is log-affine with respect to $\gamma$ when $\frac{p}{q}$ is a log-affine function.

\end{defn}

The next simple proposition characterizes log-affine random variables.

\begin{prop}

If $X$, with p.m.f. $p$, is log-affine with respect to $\gamma$, with mass function $q$, then for all $n \in \{p > 0\}$,
$$ p(n) = C \lambda^n q(n), $$
for some constants $C>0$ and $\lambda \geq 0$.

\end{prop}

\begin{proof}
One has $\{p>0\} = \llbracket a,b \rrbracket$, with $-\infty \leq a \leq b \leq +\infty$. Since $X$ is log-affine with respect to $\gamma$, we have
$$ \frac{r(n)}{r(n-1)} = \frac{r(n+1)}{r(n)}, $$
where $r(n) = p(n)/q(n)$. The ratio being constant, we deduce that $r(n) = \lambda \, r(n-1)$, where $\lambda = r(a+1)/r(a)$. Hence, $p(n) = C \lambda^n q(n)$, with $C = r(a)/\lambda^a$.
\end{proof}

We will now describe the extreme points of a class of discrete log-concave probability distributions satisfying a linear constraint. As in the continuous setting, those will be log-affine on their support.

Let us denote by $\mathcal{P}(\Z)$ the set of all probability measures supported on $\Z$. For $M, N \in \Z$, let us denote by $\mathcal{P}(\llbracket M,N \rrbracket)$ the set of all probability measures supported on $\llbracket M,N \rrbracket$. Let $\gamma$ be a measure with contiguous support on $\mathbb{Z}$, and let $h \colon \llbracket M,N \rrbracket \to \mathbb{R}$ be an arbitrary function. Let us consider $\mathcal{P}_h^{\gamma}(\llbracket M,N \rrbracket)$ the set of all distributions $\P_X$ in $\mathcal{P}(\llbracket M,N \rrbracket)$, log-concave with respect to $\gamma$, and satisfying $\E[h(X)] \geq 0$, that is,
$$ \mathcal{P}_h^{\gamma}(\llbracket M,N \rrbracket) = \{ \P_X \in \mathcal{P}(\llbracket M,N \rrbracket) : X \mbox{ log-concave with respect to } \gamma, \, \E[h(X)] \geq 0 \}. $$

Recall that the convex hull of a set $A \subset \R^d$, denoted by $\conv(A)$, is the smallest convex set containing $A$, and recall that a point $p$ is extremal in a convex set $C \subset \R^d$ if for all $\lambda \in (0,1)$  and all $x,y \in C$, if $p = (1-\lambda)x + \lambda y$ then $p=x=y$. We claim that if $\P_X$ is an extreme point of $\Conv(\mathcal{P}_h^\gamma(\llbracket M,N \rrbracket))$ then its p.m.f. $f$ with respect to $\gamma$ is of the form $f(n) = C p^n$ on a contiguous interval.

\begin{thm}\label{extreme}

    If $\P_X \in \Conv(\mathcal{P}_h^\gamma(\llbracket M,N \rrbracket))$ is an extreme point, then $f$, the p.m.f. of $X$ with respect to $\gamma$, satisfies
    \begin{equation}\label{extremizers}
        f(n) = C p^n \mathbbm{1}_{\llbracket k, l \rrbracket}(n) ,
    \end{equation}
    for some $C, p >0$, $k,l \in \llbracket M,N \rrbracket$.
    
\end{thm}

The arguments in the proof are analogous to the continuous setting (see \cite{FG04}). Before proving Theorem \ref{extreme}, we establish an intermediary lemma.

\begin{lem}\label{log-affine-conc}

If $f,g \colon \N \to [0,+\infty)$ are log-concave then the function $f\wedge g$ is log-concave, where $(f\wedge g)(n) = \min \{ f(n), g(n)\}$.  If we further assume that $g$ is log-affine,
then $(f-g)_+$ is log-concave as well, where $(f-g)_+ = \max(0, f-g)$.

\end{lem}

\begin{proof}
Clearly $f \wedge g$ has contiguous support.  Hence it suffices to prove $(f\wedge g)^2(n) \geq (f\wedge g)(n-1) (f \wedge g)(n+1)$.  Since $g^2(n) \geq   g(n-1) g(n+1) \geq  (f\wedge g)(n-1) (f \wedge g)(n+1)$, and similarly $f^2(n) \geq (f\wedge g)(n-1) (f \wedge g)(n+1)$, we have
\[
    (f\wedge g)^2(n) \geq (f\wedge g)(n-1) (f \wedge g)(n+1).
\]
Assume now that $g$ is log-affine. If $f \leq g$ there is nothing to prove, so suppose that $(f-g)(n)>0$.  If $f(n \pm 1) \leq g(n \pm 1)$ the inequality $(f-g)_+^2(n) \geq (f-g)_+(n-1)(f-g)_+(n+1)$ holds immediately.  Else, log-concavity of $f$ and affineness of $g$, 
\begin{align*}
    (f-g)(n) 
        &\geq 
            \sqrt{f(n+1)f(n-1)} - \sqrt{g(n+1)g(n-1)} 
                \\
            &\geq 
                \sqrt{(f-g)_+(n-1)(f-g)_+(n+1)},
\end{align*}
where we have used the fact that Minkowski's inequality for $L^p$ norms reverses when $p \leq 1$ and that $(x_1,x_2) \mapsto \sqrt{x_1 x_2}$ corresponds to $p=0$. It remains to show that $(f-g)_+$ has contiguous support. Let $n \geq 1$ such that $f(n-1) > g(n-1)$ while $f(n) \leq g(n)$, then for any $k \geq 1$
\begin{align*}
    \frac{g(n+k)}{g(n+k-1)} = \frac{g(n)}{g(n-1)} > \frac{f(n)}{f(n-1)} \geq \frac{f(n+k)}{f(n+k-1)},
\end{align*}
where the last inequality follows from log-concavity of $f$. Thus 
\begin{align*}
    f(n+1) = \frac{f(n+1)}{f(n)} f(n) < \frac{g(n+1)}{g(n)} f(n) \leq \frac{g(n+1)}{g(n)} g(n) = g(n+1).
\end{align*}
Inductively, it follows that for all $k \geq 0$, $f(n+k) \leq g(n+k)$. Hence, if $m,n \in \N$ are such that $m \leq n$ and $(f-g)_+(m), (f-g)_+(n) > 0$, then for all $k \in \llbracket m,n \rrbracket$, $(f-g)_+(k)>0$.
\end{proof}

\begin{proof}[Proof of Theorem \ref{extreme}]
By a translation argument, one may assume that $M=0$, thus we work on $\N$ the set of natural numbers. Suppose that $\P_X \in \Conv(\mathcal{P}_{h}^{\gamma}(\llbracket N \rrbracket))$ is an extreme point, and let $f$ be the p.m.f. of $X$ with respect to $\gamma$. Choose $k$ such that $f(k) >0$.  For $\alpha \in \mathbb{R}$ define $g_\alpha(m) =f(k) e^{\alpha(m-k)}/2$.  Since $g_\alpha$ is log-affine, the functions $(f - g_\alpha)_+$ and $f \wedge g_\alpha$ are non-zero log-concave functions by Lemma \ref{log-affine-conc}.


Note that
\begin{align} \label{eq: infinity densities}
    \lim_{\alpha \to +\infty} (f-g_\alpha)_+(m) &= \delta_{k}(m) \frac{f(k)}{2} + \mathbbm{1}_{\llbracket 0, k-1 \rrbracket}(m) f(m),
        \\
            \lim_{\alpha \to -\infty} (f-g_\alpha)_+(m) &= \delta_{k}(m) \frac{f(k)}{2} + \mathbbm{1}_{\llbracket k+1, N \rrbracket}(m) f(m),
\end{align}
while
\begin{align*}
        \lim_{\alpha \to +\infty} (f \wedge g_\alpha)(m) &= \delta_{k}(m) \frac{f(k)}{2} + \mathbbm{1}_{\llbracket k+1, N \rrbracket}(m) f(m),
        \\
            \lim_{\alpha \to -\infty} (f \wedge g_\alpha)(m) &= \delta_{k}(m) \frac{f(k)}{2} + \mathbbm{1}_{\llbracket 0, k-1 \rrbracket}(m) f(m).
\end{align*}
Let us take the above limits as the definitions of $(f-g_{\pm \infty})_+$ and $f \wedge g_{\pm \infty}$. Note also that
\begin{align}\label{equat}
f = (f - g_\alpha)_+ + f \wedge g_\alpha.
\end{align} Define, for $\alpha \in [-\infty, \infty]$, 
$X_i(\alpha)$, $i \in \{1,2\}$, as random variables with p.m.f. with respect to $\gamma$ given by
$$ d\P_{X_1(\alpha)} = C_1^{-1}(\alpha) (f- g_\alpha)_+ d \gamma, \qquad d\P_{X_2(\alpha)} = C_2^{-1}(\alpha)(f \wedge g_\alpha) d\gamma, $$
where $C_1(\alpha) = \int (f- g_\alpha)_+ d\gamma$ and $C_2(\alpha) = \int (f \wedge g_\alpha) d\gamma$.  Then by \eqref{equat}, $\P_{X}$ can be written as a convex combination of the $\P_{X_i(\alpha)}$,
\begin{align}\label{equat-PX}
    \P_{X} = C_1(\alpha) \P_{X_1(\alpha)} + C_2(\alpha) \P_{X_2(\alpha)}.
\end{align}
Observe from \eqref{eq: infinity densities} that \begin{align} \label{eq: infinity switches the mu}
\P_{X_1}(+ \infty) = \P_{X_2}(- \infty), \qquad \P_{X_1}(- \infty) = \P_{X_2}(+ \infty).
\end{align}
Define $\Psi \colon [-\infty, \infty] \to \mathbb{R}$ by
\[
    \Psi(\alpha) = \E[h(X_1(\alpha))] - \E[h(X_2(\alpha))].
\]
Note that $\Psi$ is continuous, and $\Psi(-\infty) = - \Psi(\infty)$ by \eqref{eq: infinity switches the mu}.  Thus by the intermediate value theorem, there exists $\alpha^*$ such that $\Psi(\alpha^*) =0$. Since $\E[h(X)] \geq 0$, we deduce from \eqref{equat-PX} that $\P_{X_i(\alpha^*)} \in \mathcal{P}_h^\gamma(\llbracket N \rrbracket)$.

Now, since $\P_X$ is extreme in $\Conv(\mathcal{P}_h^\gamma(\llbracket N \rrbracket))$, we have $\P_{X_1(\alpha^*)} = \P_{X_2(\alpha^*)} = \P_X$, which implies
$$ f = \frac{(f-g_{\alpha^*})_+}{C_1(\alpha^*)} = \frac{f \wedge g_{\alpha^*}}{C_2(\alpha^*)}, $$
and thus $f = C_2^{-1}(\alpha^*) g_{\alpha^*}$. Hence $X$ is log-affine with respect to $\gamma$.
\end{proof}





\begin{remark}


\textbullet \, Note that on the support of an extreme point $\P_X \in \Conv(\mathcal{P}_h^{\gamma}(\llbracket N \rrbracket))$, with p.m.f. $p$, the function $\Lambda(x) = \sum_{n=0}^x h(n) p(n)$ must never switch signs. If $h$ is of constant sign, then this is obvious. Assume $h$ is not of constant sign, and assume without loss of generality that there exists $k \in \llbracket N - 1 \rrbracket$ such that $\Lambda(k) \geq 0$ and $\Lambda(k+1) < 0$, then define for $t \in [0,1]$ and $n \in \llbracket N \rrbracket$, 
\begin{align*}
    p_{1,t}(n) &= \frac{p(n) 1_{\llbracket 0 ,k \rrbracket}(n) + t p(k+1) \delta_{k+1}(n)}{\P_X( \llbracket 0 , k \rrbracket) + t p(k+1)},
        \\
    p_{2,t}(n) &= \frac{p(n) 1_{\llbracket k+2 ,N \rrbracket}(n) + (1-t) p(k+1) \delta_{k+1}(n)}{\P_X( \llbracket k+2 , N \rrbracket) + (1-t) p(k+1)}.
\end{align*} 
Note that $\P_X$ must give positive measure to $\llbracket k+2, N \rrbracket$ or else $0 >\Lambda(k+1) = \Lambda(N) = \E[h(X)]$, which is a contradiction.  Now define $\Psi(t) = \sum_{n=0}^N h(n) p_{1,t}(n)$.  By the conditions on $\Lambda$, $\Psi(0) \geq 0$ while $\Psi(1) < 0$, thus there exists $t^* \in [0,1]$ such that $\Psi(t^*) =0$.  From this we can split $\P_X$ as
$$ \P_X = (1-\lambda) \P_{X_1} + \lambda \P_{X_2}, $$
where $X_1$ has p.m.f. $p_{1,t^*}$, $X_2$ has p.m.f. $p_{2,t^*}$, and $\lambda = \P_X( \llbracket k+2 , N \rrbracket) + (1-t) p(k+1) \in (0,1)$. Since $\P_{X_1}, \P_{X_2} \in \mathcal{P}_h^{\gamma}(\llbracket N \rrbracket)$, this is a contradiction to the extremality of $\P_X$.

\vskip3mm
\textbullet \, Let us also note that a non-Dirac extreme point $\P_X \in \Conv(\mathcal{P}_h^{\gamma}(\llbracket N \rrbracket))$ satisfies
$$ \E[h(X)] = 0. $$
Indeed, denote $\Lambda(x) = \sum_{n=0}^x h(n) p(n)$ for $x \in \llbracket N \rrbracket$, and assume towards a contradiction that $\Lambda(N) = \E[h(X)] > 0$. Denote by $m$ the smallest element in $\llbracket N \rrbracket$ such that $\Lambda(m) > 0$. By the previous remark, $\Lambda \geq 0$, hence for all $x<m$, $\Lambda(x) = 0$. It follows that $\Lambda(m) = p(m) h(m) > 0$, and thus $p(m) > 0$. Now, define for $t \in (0,1)$,
\begin{align*}
    p_{1,t}(n) &= \frac{p(n) 1_{\llbracket 0 , m-1 \rrbracket}(n) + t p(m) \delta_{m}(n)}{\P_X( \llbracket 0 , m-1 \rrbracket) + t p(m)},
        \\
    p_{2,t}(n) &= \frac{p(n) 1_{\llbracket m+1 ,N \rrbracket}(n) + (1-t) p(m) \delta_{m}(n)}{\P_X( \llbracket m+1 , N \rrbracket) + (1-t) p(m)},
\end{align*}
and we can split $\P_X$ for $t$ close enough to $0$.

\end{remark}

Theorem \ref{extreme} tells us that if we want to maximize a convex function over $\mathcal{P}_h^{\gamma}(\llbracket M,N \rrbracket)$, it is enough to check probability distributions that are log-affine on a segment:

\begin{coro}\label{maximize}

Let $\Phi \colon \mathcal{P}_h^{\gamma}(\llbracket M,N \rrbracket) \to \R$ be a convex function. Then
$$ \sup_{\P_X \in \mathcal{P}_h^{\gamma}(\llbracket M,N \rrbracket)} \Phi(\P_X) \leq \sup_{\P_{X} \in \mathcal{A}_h^{\gamma}(\llbracket M,N \rrbracket)} \Phi(\P_{X}), $$
where $\mathcal{A}_h^{\gamma}(\llbracket M,N \rrbracket)$ is the subset of $\mathcal{P}_h^{\gamma}(\llbracket M,N \rrbracket)$ whose p.m.f. with respect to $\gamma$ is of the form $f(n) = C p^n \mathbbm{1}_{\llbracket k, l \rrbracket}(n)$, for some $C,p > 0$ and $k,l \in \llbracket M,N \rrbracket$.


\end{coro}

In the next corollary we demonstrate that the identification of extreme points can be used to derive a discrete analogue of the original localization lemma of Lov\'asz and Simonovits \cite{KLS95}. First, introduce the notations
\begin{eqnarray}\label{log-c}
\mathcal{L}(\gamma) & = &
\{ \mu \in \mathcal{P}(\Z), \mu \mbox{ log-concave with respect to } \gamma \}, \\ \label{log-a}
\mathcal{A}(\gamma) & = &
\{ \mu \in \mathcal{P}(\Z), \mu \mbox{ log-affine with respect to } \gamma \}.
\end{eqnarray}

\begin{cor} \label{cor: positivity condition}

For $\mu$ a measure on $\mathbb{Z}$, log-concave with respect to a reference measure $\gamma$, and $f,g \colon \mathbb{Z} \to \mathbb{R}$, belonging to $\ell_1(\mu)$ such that 
\[
    \sum_i f(i) \mu(i) > 0 \hbox{ and } \sum_i g(i) \mu(i) > 0,
\]
 there exists a probability measure $\nu$ with finite support, log-affine with respect to $\gamma$, such that
\begin{align*} \label{eq: affine conclusion}
    \sum_i f(i) \nu(i) > 0 \hbox{ and } \sum_i g(i) \nu(i) > 0.
\end{align*}

\end{cor}

\begin{proof}
   For large enough $M$, 
   \[
         \sum_{i= -M}^M f(i) \tilde{\mu}(i) > 0 \hbox{ and } \sum_{i= -M}^M g(i) \tilde{\mu}(i) > 0,
   \]
 where $\tilde{\mu}$ is the normalized restriction of $\mu$ to $\llbracket - M, M \rrbracket$, explicitly $\tilde{\mu}(A) \coloneqq \frac{\mu(A \cap \llbracket - M, M \rrbracket )}{\mu( \llbracket -M, M \rrbracket)}$.  Taking $h = g - \sum_{i = -M}^{M} g(i) \tilde{\mu}(i)$, one has $\tilde{\mu} \in \mathcal{P}_{h}^\gamma( \llbracket - M, M \rrbracket)$, and by Corollary \ref{maximize} the maximum of $\Phi \colon \mathcal{P}_{h}^\gamma( \llbracket - M, M \rrbracket) \to \mathbb{R}$, defined as
 \[
    \Phi(\sigma)  = \sum_{i = -M}^M f(i) \sigma(i),
 \]
occurs at an extreme point $\nu$. In this case 
\[
    \sum_{i = -M}^M h(i) \nu(i) = \sum_{i= -M}^M g(i) \nu(i) - \sum_{i = -M}^M g(i) \tilde{\mu}(i) \geq 0
\]
so that $\sum_{i = - M}^M g(i) \nu(i) > 0$.  Further 
\[
    \Phi(\nu) \geq \Phi(\tilde{\mu}) >0
\]
implies that $\sum_{i= -M}^M f(i) \nu(i) > 0$ as well.
\end{proof}

\section{Applications} \label{sec: Applications}

In this section, we discuss applications of our techniques in the discrete setting.

\subsection{A Discrete Analogue of the KLS Lemma for Four Functions}

\begin{thm}\label{thm: four function theorem}

Given $f_1,f_2, f_3,f_4$ nonnegative functions, and $\alpha, \beta >0$, then the inequality
%
%
\begin{equation}\label{4-funct}
\E[f_1(X)]^\alpha \E[f_2(X)]^\beta \leq  \E[f_3(X)]^\alpha \E[f_4(X)]^\beta
\end{equation}
holds for all $X$ log-concave random variable with respect to $\gamma$ if and only if it holds for all log-affine random variable with respect to $\gamma$ on a segment.


\end{thm}

\begin{proof}

One direction is immediate. For the other direction, given $X$ log-concave with respect to $\gamma$, it is enough to prove that $\E[f_1(X)]^\alpha \E[f_2(X)]^\beta \leq  \left( \E[f_3(X)] + \varepsilon \right)^\alpha \E[f_4(X)]^\beta$
    holds for all $\varepsilon> 0$. By an approximation argument, one may assume that $X$ is compactly supported, say on $\llbracket M,N \rrbracket$.  Writing $\tilde{f}_3 = f_3 + \varepsilon$, and
    \begin{align*}
        \Phi(\P_Z) = \left( \frac{ \E[f_1(X)] }{ \E[\tilde{f}_3(X)]} \right)^{\frac{\alpha}{ \beta}} \E[f_2(Z)] -  \E[f_4(Z)],
    \end{align*}
    we wish to show that $\Phi(\P_X) \leq 0$. Defining $h = \E[\tilde{f}_3(X)] f_1 - \E[f_1(X)] \tilde{f}_3$, for every $\P_Y \in \mathcal{P}_h^{\gamma}(\llbracket M,N \rrbracket)$ log-affine with respect to $\gamma$, one has
    \begin{eqnarray*}
     \Phi(\P_Y) & = & \left( \frac{ \E[f_1(X)] }{ \E[\tilde{f}_3(X)]} \right)^{\frac{\alpha}{ \beta}} \E[f_2(Y)] -  \E[f_4(Y)] \\ & \leq & \left( \frac{ \E[f_1(Y)] }{ \E[\tilde{f}_3(Y)]} \right)^{\frac{\alpha}{ \beta}} \E[f_2(Y)] -  \E[f_4(Y)] \\ & \leq & 0,
     \end{eqnarray*}
     where the first inequality comes from the fact that $\E[h(Y)] \geq 0$ and the second inequality from the fact that \eqref{4-funct} holds for all log-affine distribution. Since $\P_X \in \mathcal{P}_h^{\gamma}(\llbracket M,N \rrbracket)$,we deduce by Corollary \ref{maximize} that $\Phi(\P_X) \leq 0$.
\end{proof}

The next result is a consequence of Theorem \ref{thm: four function theorem} and tells us that the class of discrete log-concave distributions with respect to a reference measure is closed under convolution if and only if the convolution of log-affine distributions are log-concave with respect to that reference measure. Recall that for functions $f,g \colon \Z \to \R$, their convolution is defined as
$$ (f * g) (k) = \sum_{n=-\infty}^{+\infty} f(n) g(k-n), \qquad k \in \Z. $$
For $A,B \subset \mathcal{P}(\Z)$, we will use the notation $A * B = \{\mu * \nu : \mu \in A, \nu \in B\}$, where for $\mu$ with p.m.f. $f$ and $\nu$ with p.m.f. $g$, $(\mu * \nu)(\{k\}) = (f*g)(k)$ for all $k \in \Z$. Recall the notations \eqref{log-c} and \eqref{log-a}.

\begin{cor}\label{cor: logconcave stability under convolution reduced to log-affine case}

We have $\mathcal{L}(\gamma) * \mathcal{L}(\gamma) \subseteq \mathcal{L}(\gamma)$ if and only if $\mathcal{A}(\gamma)*\mathcal{A}(\gamma) \subseteq \mathcal{L}(\gamma)$.
    
\end{cor}

\begin{proof}
Denote by $q$ the mass function of $\gamma$. Suppose that $\mathcal{A}(\gamma)*\mathcal{A}(\gamma) \subseteq \mathcal{L}(\gamma)$, we will first show that $\mathcal{L}(\gamma)*\mathcal{A}(\gamma) \subseteq \mathcal{L}(\gamma)$.  Given $\mu \in \mathcal{A}(\gamma)$ with p.m.f. $f$ and $\nu \in \mathcal{L}(\gamma)$ with p.m.f. $g$, we wish to show that for a fixed $k$
\begin{align} \label{eq: log-concavity of convolution equation}
    \left(\frac{f*g}{q}\right)^2(k) \geq \frac{f*g}{q}(k+1) \frac{f*g}{q}(k-1).
\end{align}
Define $f_1(x) = f_2(x) = f(k-x)$, $f_3(x) = \frac{q^2(k)}{q(k+1)q(k-1)} f(k+1-x)$, $f_4(x) = f(k-1-x)$ and $\alpha = \beta = 1$, then \eqref{eq: log-concavity of convolution equation} is equivalent to
\begin{align} \label{eq: log-concavity of convolution equation as four function}
\E[f_1(Y)] \E[f_2(Y)] \geq  \E[f_3(Y)] \E[f_4(Y)],
\end{align}
and since \eqref{eq: log-concavity of convolution equation} holds whenever $g$ is log-affine with respect to $\gamma$, \eqref{eq: log-concavity of convolution equation as four function} holds whenever $Y$ is log-affine as well. Thus by Theorem \ref{thm: four function theorem}, \eqref{eq: log-concavity of convolution equation as four function} holds for all $Y$ log-concave with respect to $\gamma$, equivalently, \eqref{eq: log-concavity of convolution equation} holds for all $g \in \mathcal{L}(\gamma)$. Thus $f*g \in \mathcal{L}(\gamma)$ if $f, g \in \mathcal{L}(\gamma)$ and at least one of $f$ and $g$ is an element of $\mathcal{A}(\gamma)$.  Repeating the same argument assuming only that $f \in \mathcal{L}(\gamma)$ completes the proof.
\end{proof}

We can thus give a direct simple computational argument of the fact that log-concave sequences are stable under convolution \cite{fekete1912problem}, as well as Walkup's theorem on the stability under convolution of ultra log-concave sequences \cite{Wal76}.

\begin{coro}[Fekete \cite{fekete1912problem}]

For $p$ and $q$ log-concave sequences, $p*q$ is log-concave as well.

\end{coro}

\begin{proof}
By Corollary \ref{cor: logconcave stability under convolution reduced to log-affine case} and homogeneity it suffices to prove the result when $p(n) = p^n \mathbbm{1}_A(n) $ and $q(m) = q^m \mathbbm{1}_B(m)$ where $p, q > 0$ and $A, B$ are integer intervals. One can express $(p*q)^2(k) \geq (p*q)(k+1) \, (p*q)(k-1)$ as
$$ \sum_{n,m} p(n)p(m) q(k-n)q(k-m) \geq \sum_{n,m} p(n) p(m) q(k+1-n) q(k-1-m), $$
which after change of variable $l = m+n$ gives
\begin{align*}
    q^{2k} \sum_{l} &\left( \frac{p}{q} \right)^l \sum_{n} \mathbbm{1}_{A' \cap B'}(n)
    \geq q^{2k} \sum_{l} \left( \frac{p}{q} \right)^l \sum_n \mathbbm{1}_{A' \cap (B'+1)}(n),
\end{align*}
where $A' = A \cap [l -A] $ and $B' = [k - B] \cap [l - (k-B)]$. The result then follows from observing that $A'$ and $B'$ are both intervals symmetric about $l/2$.
\end{proof}

\begin{coro}[Walkup \cite{Wal76}]\label{Walkup}

For $p$ and $q$ ultra log-concave sequences, $p*q$ is ultra log-concave as well.

\end{coro}

\begin{proof}
By Corollary \ref{cor: logconcave stability under convolution reduced to log-affine case} and homogeneity, it suffices to consider
$$ p(n) = \frac{p^n}{n!}1_{[a,b]}(n), \quad q(n) = \frac{q^n}{n!}1_{[c,d]}(n). $$
Their convolution is
$$ (p * q)(k) = \frac{q^n}{n!} h(k), $$
where
$$ \qquad h(k) := \sum_{n = a \vee (k-d)}^{b \wedge (k-c)} \binom{k}{n} R^n, $$
and $R = p/q$. To prove that $p * q$ is ultra log-concave, it is enough to check that $h$
is log-concave in $k$, for all $R>0$, $a,b,c,d \in \N$, with the convention $\binom{k}{n} = 0$ if $n<0$ or $n>k$. By the commutative property of convolution, $p * q = q * p$, the log-concavity of $h$ can be reduced to the following 3 inequalities,
\begin{eqnarray}\label{walk-ineq}
\left[ \sum_{n=a}^{b} \binom{k}{n} R^n \right]^2 & \geq & \sum_{n=a}^{b+1} \binom{k+1}{n} R^n \,\, \sum_{n=a}^{b-1} \binom{k-1}{n} R^n, \nonumber \\ \nonumber
\left[ \sum_{n=a}^{b} \binom{k}{n} R^n \right]^2 & \geq & \sum_{n=a+1}^{b} \binom{k+1}{n} R^n \,\, \sum_{n=a-1}^{b} \binom{k-1}{n} R^n, \\
\left[ \sum_{n=a}^{b} \binom{k}{n} R^n \right]^2 & \geq & \sum_{n=a+1}^{b+1} \binom{k+1}{n} R^n \,\, \sum_{n=a-1}^{b-1} \binom{k-1}{n} R^n.
\end{eqnarray}
The above three inequalities can all be established in the same way via the elementary Pascal identity
\begin{equation}\label{pascal}
\binom{k}{n} = \binom{k-1}{n} + \binom{k-1}{n-1}.
\end{equation}
For the reader convenience, we present the argument for the last inequality \eqref{walk-ineq}. Let us denote, for $0 \leq a \leq b \leq k$,
\begin{equation*}
    h_{a,b}(k) = \sum_{n=a}^{b} \binom{k}{n} R^n.
\end{equation*}
Using Pascal identity \eqref{pascal}, note that
\begin{equation}\label{rec}
h_{a,b}(k) = (R+1)h_{a-1,b-1}(k-1) + \binom{k-1}{b} R^b - \binom{k-1}{a-1} R^{a-1}.
\end{equation}
Therefore,
\begin{eqnarray*}
& ~ & h_{a,b}(k)^2 - h_{a+1,b+1}(k+1) h_{a-1,b-1}(k-1) \\
 & = & h_{a,b}(k) \left[ (R+1)h_{a-1,b-1}(k-1) + \binom{k-1}{b}R^b - \binom{k-1}{a-1}R^{a-1} \right] \\ & ~ & \hspace*{3cm} - \, h_{a-1,b-1}(k-1) \left[ (R+1)h_{a,b}(k) + \binom{k}{b+1}R^{b+1} - \binom{k}{a}R^{a} \right] \\ & = & h_{a,b}(k) \left[ \binom{k-1}{b}R^b - \binom{k-1}{a-1}R^{a-1} \right] - \, h_{a-1,b-1}(k-1) \left[ \binom{k}{b+1}R^{b+1} - \binom{k}{a}R^{a} \right] \\
 & = & \sum_{n=a}^{b} \left[ \binom{k-1}{n-1} \binom{k}{a} - \binom{k}{n} \binom{k-1}{a-1} \right] R^{n+a-1} + \sum_{n=a}^{b} \left[ \binom{k}{n} \binom{k-1}{b} - \binom{k-1}{n-1} \binom{k}{b+1} \right] R^{n+b}.
\end{eqnarray*}
Since $a \leq n \leq b$, each term in the summation is non-negative by direct calculation.
\end{proof}

To emphasize the strength of our techniques, we present a simple proof of the following Walkup-type theorem established by Klartag and Lehec in \cite{KL19}. Recall that a function $f \colon \mathbb{Z} \to [0,\infty)$ is unimodal when $m \leq k \leq n$ implies 
\begin{align*}
    f(k) \geq \min \{ f(m), f(n) \}.
\end{align*}

\begin{coro}[Klartag-Lehec \cite{KL19}]\label{KL}

If $\{a_k\}_{k \geq 0}$ is a log-concave sequence then the sequence $\{c_k\}_{k \geq 0}$ defined by
$$ c_k = \sum_{n \geq k} \binom{n}{k}a_n $$
is log-concave as well.

\end{coro}

\begin{proof}
Recall the convention $\binom{n}{k} = 0$ if $k<0$ or $k>n$. We also use the convention that $\binom{n}{k} = 0$ if $n$ is not an integer. By Theorem \ref{thm: four function theorem} and homogeneity, it is enough to consider $a_n = p^n1_{[a,b]}(n)$. In this case, after a change of variable, the inequality $c_k^2 \geq c_{k-1} c_{k+1}$ can be expressed as
\begin{equation*}
\sum_{l \in \N} p^l \left[ \sum_{n =a \vee (l-b)}^{b \wedge (l-a)} \binom{n}{k}\binom{l-n}{k} \right] \geq \sum_{l \in \N} p^l \left[ \sum_{n =a \vee (l-b)}^{b \wedge (l-a)} \binom{n}{k-1}\binom{l-n}{k+1} \right].
\end{equation*}
Fixing $l \geq 0$ and assuming without loss of generality that $a \geq l-b$, it is enough to prove that
$$ F(a) := \sum_{n =a }^{l-a} \left[ \binom{n}{k}\binom{l-n}{k} - \binom{n}{k-1}\binom{l-n}{k+1} \right] \geq 0. $$
Note that $F$ is constant outside the interval $\{k-1, \dots, l/2\}$. We will conclude that $F \geq 0$ by showing that $F$ is unimodal, $F(k-1) = 0$ and $F(l/2) \geq 0$.
Denote
$$ f(n) = \binom{n}{k}\binom{l-n}{k} - \binom{n}{k-1}\binom{l-n}{k+1}, \qquad g(n) = f(n) + f(l-n). $$
First, note that if $n \geq \frac{l-1}{2}$, then $f(n) \geq 0$.
Hence, $F(l/2) = f(l/2) \geq 0$.
On the other hand, applying the Chu-Vandermonde type identity
$$ \sum_{n=0}^l \binom{n}{k}\binom{l-n}{s-k} = \binom{l+1}{s+1}, $$
we have
$$ \sum_{n =k-1 }^{l-k+1} \binom{n}{k}\binom{l-n}{k} = \sum_{n = 0 }^{l} \binom{n}{k}\binom{l-n}{k} = \binom{l+1}{2k+1} = \sum_{n = k-1 }^{l-k+1} \binom{n}{k-1}\binom{l-n}{k+1}. $$
Therefore, $F(k-1)=0$. Next, note that
\begin{eqnarray*}
    g(n) \geq 0 & \Longleftrightarrow & 2 \binom{n}{k}\binom{l-n}{k} - \binom{n}{k-1}\binom{l-n}{k+1} - \binom{l-n}{k-1}\binom{n}{k+1} \geq 0 \\ & \Longleftrightarrow & an^2 +bn + c \geq 0,
\end{eqnarray*}
with $a = -2\frac{k+1}{k} - 2 < 0$. Since $g(n) = g(l-n)$, we deduce the existence of $n_0 \leq l/2$ such that $g(n) < 0$ for $n \leq n_0$ and $n \geq l-n_0$, and $g(n) \geq 0$ for $n \in \{n_0+1, ..., l - n_0-1\}$. In fact, $n_0 \geq k-1$ since we have seen that $F(k-1) = 0$ and $F(l/2) \geq 0$. It follows that $F(a)$ increases when $a \in \{k-1, ..., n_0\}$ and $F(a)$ decreases when $a \in \{n_0+1, ..., l/2\}$, therefore $F$ is unimodal.
\end{proof}

\begin{remark}

The above proof of Corollary \ref{KL} can be adapted to provide yet a second simple proof of Walkup's theorem. The slight difference consists in showing that 
$$ \widetilde{F}(a) := \sum_{n =a }^{l-a} \left[ \binom{k}{n}\binom{k}{l-n} - \binom{k-1}{n}\binom{k+1}{l-n} \right] \geq 0, $$
which can be done with similar considerations via the standard Vandermonde identity.

\end{remark}


The following inequality due to Gurvits was recently utilized by Havrilla, Nayar, and Tkocz in \cite{havrilla2021khinchin}, where an alternative and elementary proof is also given, to prove that for $X$ a symmetric random variable,
 being strongly log-concave, in the sense that $f_X(t) \coloneqq \mathbb{E} e^{\sqrt{t}X}$ and all of $f_X$'s derivatives are log-concave in $t$, is equivalent to $X$ being ultra sub-Gaussian in the sense of \cite{Nayar2012khinchine}. Theorem \ref{thm: four function theorem} allows an even simpler proof.

\begin{coro}[Gurvits \cite{gurvits2009multivariate}]\label{Gur}
If $(p_n)$ is a log-concave sequence, then the power series $f(t) \coloneqq \sum_{n=0} \frac{p_n}{n!}t^n$ is log-concave in $t \in (0,\infty)$.
\end{coro}

\begin{proof}
By approximation, it suffices to consider $(p_n)$ finite, and to prove that the power series $F \coloneqq - f^2 (\log f )'' =  (f')^2- f f''  \geq 0$. Using Theorem \ref{thm: four function theorem} it suffices to take $p_n = p^n \mathbbm{1}_{A}(n)$, where $A = \llbracket a, b \rrbracket$.  We will prove that $F$ is non-negative by proving that its coefficients are non-negative. Indeed,
 \begin{align} \label{eq: the powerseries expression}
    F(t) = \sum_{l = 0}^\infty t^l \left[\frac{p^{l + 2}}{l!}  \sum_{n \in \mathbb{Z}} \binom{l}{n} \bigg( \mathbbm{1}_A(n+1) 
    \mathbbm{1}_A(l -n + 1) - \mathbbm{1}_A(n) \mathbbm{1}_A(l -n + 2) \bigg) \right].
 \end{align}
We observe that the coefficients  of \eqref{eq: the powerseries expression} are zero unless $m \coloneqq (a-1) \vee (l-(b-1)) \leq \frac l 2$, in which case,
\begin{align*}
   \sum_{n \in \mathbb{Z}} \binom{l}{n} \bigg( \mathbbm{1}_A(n+1) \mathbbm{1}_A(l -n + 1) - \mathbbm{1}_A(n) \mathbbm{1}_A(l -n + 2) \bigg) 
   = \binom{l}{m} - \binom{l}{m-1} \geq 0.
\end{align*}
\end{proof}

\subsection{Discrete Pr\'ekopa-Leindler Inequality}\label{Prekopa}

The classical Pr\'ekopa-Leindler inequality (\cite{Pre71}, \cite{Lei72a}) states that if $t \in [0,1]$ and $f,g \colon \R^d \to [0, +\infty)$ are non-negative measurable functions such that
$$ f \square_t g(z) = \sup_{\{(x,y) \in \R^{2d} \, : \, z = (1-t)x + t y\}} f^{1-t}(x) g^t(y) $$
is measurable, then
\begin{align}\label{pl}
    \int_{\R^d} f \square_t g(z) dz \geq \left( \int_{\R^d} f(z) dz \right)^{1-t} \left( \int_{\R^d} g(z) dz \right)^{t}.
\end{align}

The Brunn-Minkowski inequality, fundamental in convex geometry (see, e.g., \cite{Gar02}), is a geometric analogue of the Pr\'ekopa-Leindler inequality and states that if $t \in [0,1]$ and $A,B \subset \R^d$ are measurable sets such that $(1-t)A + t B$ is measurable, then
\begin{align}\label{bm}
|(1-t)A + t B|_d^{\frac{1}{d}} \geq (1-t)|A|_d^{\frac{1}{d}} + t |B|_d^{\frac{1}{d}}.
\end{align}

There has been impressive recent progress on Pr\'ekopa-Leindler and Brunn-Minkowski type inequalities on lattices. 
For example, Halikias-Klartag-Slomka \cite{HKS} proved that
    \begin{align}\label{KL-prek}
       \left( \sum_x f(x) \right) \left( \sum_x g(x) \right) \leq \left(\sum_x h(x) \right) \left( \sum_x k(x) \right)
    \end{align}
for all functions $f,g,h,k : \mathbb{Z}^d \to [0,\infty)$ and $t \in [0,1]$ such that
    \begin{align} \label{eq: other-pl hypothesis}
        f(x) g(y) \leq h( \lfloor t x + (1-t) y \rfloor ) k( \lceil (1-t) x+ t y \rceil) 
    \end{align}
holds for $x, y \in \mathbb{Z}^d$, generalizing a result of Klartag-Lehec \cite{KL19} (see also \cite{GRST19}, \cite{S20}). Another discrete Pr\'ekopa-Leindler type inequality was derived by Iglesias-Yepes Nicol\'as-Zvavitch in \cite{LNZ20}: Let $t \in (0,1)$ and let $K, L \subset \R^d$ be non-empty bounded sets. Let $f, g, h \colon \R^d \to [0, +\infty)$ be functions such that
$$ h((1-t)x + t y) \geq f(x)^{1-t} g(y)^{t}, $$
for all $x \in K$ and $y \in L$. Then,
\begin{align}\label{other-pl}
    \sum_{z \in M \cap \Z^d} h^{\circ}(z) \geq \left(\sum_{z \in K \cap \Z^d} f(z)\right)^{1-t} \left( \sum_{z \in L \cap \Z^d} g(z) \right)^{t},
\end{align} 
where $M = (1 - t)K + tL + (-1, 1)^n$, and $h^{\circ}(z) = \sup_{s \in (-1,1)^n} h(z+s)$.

In the literature, these inequalities are referred to as Brunn-Minkowski inequalities for the geometric type results that can be deduced. However, a motivation going back to \cite{OV12, GRST14} for developing Pr\'ekopa-Leindler type inequalities in the discrete setting, and conspicuously absent currently, was the potential to derive functional inequalities from these ``geometric inequalities''. In the continuous setting this is done through differentiating the Pr\'ekopa-Leindler inequality at $t = 0$, requiring sharp control of the inequality for $t \to 0$.  Analogously, the Euclidean isoperimetric inequality for example can be derived easily by differentiating the Brunn-Minkowski inequality. Our contribution is somewhat modest in comparison to recent developments, as it only applies to the one-dimensional case, and only to unimodal functions, however it does capture the correct behavior for small $t$ in cases that elude the previously derived inequalities, as we will demonstrate at the end of this section.

\begin{thm} \label{thm: main}
Suppose that $f$ and $g$ are unimodal $\ell^1(\mu)$ functions for $\mu$ log-concave, then
\begin{align}\label{our-pl}
    \int f \square_t g(z) d\mu(z) \geq \left( \int f(z) d\mu(z) \right)^{1-t} \left( \int g(z) d \mu(z) \right)^{t},
\end{align}
where
$$ f \square_t g(z) = \sup_{\{(x,y) \in \Z^2 \, : \, |(1-t)x + t y - z| < 1 \}} f^{1-t}(x) g^t(y). $$
\end{thm}

\begin{proof}[Proof of Theorem \ref{thm: main}]
We will first prove the result in the special case that $f$ and $g$ are indicators.  Since $f$, $g$ are unimodal indicator functions they can be written as $f = \mathbbm{1}_{ \llbracket a_1, a_2 \rrbracket}$ and $g = \mathbbm{1}_{\llbracket b_1 , b_2 \rrbracket}$ for intervals contained in the support of $\mu$.  In this case we can write
\begin{align*}
    f \square_t g(z) = \mathbbm{1}_{ \llbracket L_1,  L_2 \rrbracket }
\end{align*}
with $L_1 = \lf (1-t) a_1 + t b_1 \rf$ and $L_2 = \lc (1-t) a_2 + t b_2 \rc$.  To prove that $
    \int f \square_t g(z) d\mu(z) \geq \left( \int f(z) d\mu(z) \right)^{1-t} \left( \int g(z) d \mu(z) \right)^{t},
$
if one applies Theorem \ref{thm: four function theorem} with $\alpha = 1-t, \beta = t$, $f_1 = f, f_2 = g,$ and $f_3= f_4 = f \square_t g$ then it suffices to prove the result when $\mu = \nu$ is a log-affine measure. By normalizing and translating, we may assume that $a_1, b_1 \geq 0$ and that $\nu(k) = p^k$ for $p \in (0,1]$.  Note that if $p = 1$, the proof is an immediate computation, that can alternatively be recovered from the $p \in (0,1)$ case, thus we further assume $p < 1$. In this case we have
\begin{align*}
    \left( \int f(z) d\nu(z) \right)^{1-t} \left( \int g(z) d \nu(z) \right)^{t}
        &=
            \left(\sum_{j=a_1}^{a_2} p^j \right)^{1-t} \left( \sum_{j=b_1}^{b_2} p^j \right)^{t}
                \\
        &=
            p^{a_1(1-t)+b_1 t} \frac{ (1 - p^{a_2 - a_1 + 1})^{1-t}(1 - p^{b_2 - b_1 + 1})^t}{1-p}
                \\
        &\leq
            p^{a_1(1-t)+b_1 t} \frac{ (1-t) (1 - p^{a_2 - a_1 + 1}) + t(1 - p^{b_2 - b_1 + 1})}{1-p}
                \\
        &\leq
            p^{a_1(1-t)+b_1 t} \frac{ 1 - p^{(1-t)a_2+tb_2 - (1-t)a_1-tb_1 + 1}}{1-p}
                \\
        &\leq
            p^{L_1} \frac{ 1 - p^{L_2 - L_1 + 1}}{1-p}
                \\
        &=
            \int f \square_t g(z) d \nu.
\end{align*}
The first two inequalities are by the arithmetic-geometric means inequality and the third is by monotonicity. \\

Now let us assume that $f$ and $g$ take finitely many values all belonging to the support of $\mu$.  In this case by unimodality $f = \sum_{i=1}^n f_i \mathbbm{1}_{A_i}$ for $f_i > 0$ and $A_i$ intervals such that $A_i \subseteq A_{i-1}$ while $g = \sum_{j=1}^m g_j \mathbbm{1}_{B_j}$ for $g_j >0$ and $B_j$ intervals such that $B_j \subseteq B_{j-1}$.  We proceed by induction, with the case $m+n \leq 2$ complete we may assume that $m + n = k$ and that the desired inequality holds for functions $\tilde{f}$ and $\tilde{g}$ satisfying $\tilde{m} + \tilde{n} < k$. Without loss of generality, we may assume that $n \geq 2$. Define $F_B = \sum_{i=1}^{n-1} f_i \mathbbm{1}_{A_i}$ and $F_T =f_n \mathbbm{1}_{A_n} $, so that $\int F_B(z) d \mu(z) < \int f(z) d\mu(z)$.  Now define $G_B^{(\lambda)}  = \min \{ g, \lambda \}$, and by the intermediate value theorem, since $\int G_B^{(0)}(z) d\mu(z) = 0$ and $\lim_{\lambda \to \infty} \int G_B^{(\lambda)}(z) d\mu(z) = \int g(z) d \mu(z)$ there exists $\lambda_0 \in (0,\infty)$ such that
\begin{align*}
    \int G_B^{(\lambda_0)}(z) d \mu(z) = \frac{\int F_B(z) d \mu(z) } {\int f(z) d\mu(z)} \int g(z) d\mu(z).
\end{align*}
Define $G_B = G_B^{\lambda_0}$ and $G_T = g - G_B$.  We now claim that
\begin{align*}
    f \square_t g \geq F_B \square_t G_B  + F_T \square_t G_T.
\end{align*}
To show this, observe that $\supp(F_T) \subseteq \{ z : F_B(z) = \| F_B \|_\infty \}$ and $\supp(G_T) \subseteq \{ z : G_B(z) = \| G_B \|_\infty \}$, and that if $F_T \square_t G_T(z) = 0$ the result is immediate, as $f \geq F_B$ and $g \geq G_B$ will imply $f \square_t g(z) \geq F_B \square_t G_B(z)$.  Thus if $F_T \square_t G_T(z) >0$ then there exist $x,y$ such that $|(1-t) x + t y - z| < 1$  and $F_T^{1-t}(x) G_T^{1-t}(y) = F_T \square_t G_T(z)$.  Further, $x$ and $y$ belong to the respective supports of $F_T$ and $G_T$, $F_B(x) = \|F_B\|_\infty$ and $G_B(y) = \|G_B\|_\infty$, so that $F_B  \square_t G_B(z)=  F_B^{1-t}(x) G_B^t(y)$.  Computing,
\begin{align*}
   F_B \square_t G_B(z) + F_T \square_t G_T(z)
    &=
        F_B^{1-t}(x) G_B^t(y) + F_T^{1-t}(x) G_T^t(y)
            \\
    &\leq
        (F_B(x)+ F_T(x) )^{1-t} (G_B(y) + G_T(y))^t
            \\
    &=
        f^{1-t}(x) g^t(y)
            \\
    &\leq
        f \square_t g (z).
\end{align*}
Thus
\begin{align*}
    \int f \square_t g (z) dz \geq \int F_B \square_t G_B(z)d\mu(z) + \int F_T \square_t G_T(z) d \mu(z).
\end{align*}
Observe that $F_B, G_B, F_T$ and $G_T$ are all unimodal.  Moreover, since $F_B$ and $F_T$ can both be expressed in terms of a summation of nested indicators with strictly fewer than $n$ terms, and $G_B$ and $G_T$ can both be expressed in terms of a summation of nested indicators with no more than $m$ terms, both integrals satisfy the inductive hypothesis so that
\begin{align*}
    \int F_B \square_t G_B(z)d\mu(z) &+ \int F_T \square_t G_T(z) d \mu(z)
        \\
        &\geq
            \left( \int F_B d \mu \right)^{1-t} \left( \int G_B d \mu \right)^{t} + \left( \int F_T d \mu \right)^{1-t} \left( \int G_T d \mu \right)^{t}
                \\
        &=
            \left( \int f d \mu \right)^{1-t} \left( \int g d \mu \right)^{t}.
\end{align*}

The case of arbitrary unimodal $f,g$ is completed by considering $\tilde{f}$, $\tilde{g}$ unimodal functions taking finitely many values such that  $0 \leq \tilde{f} \leq f$ and $0 \leq \tilde{g} \leq g$, and observing that $f\square_t g \geq \tilde{f} \square_t \tilde{g}$, so that
\[
    \int f \square_t g(z) d\mu(z) \geq \int \tilde{f} \square_t \tilde{g}(z) d\mu(z) \geq \left( \int \tilde{f} d\mu \right)^{1-t} \left( \int \tilde{g} d\mu \right)^t.
\]
Taking the supremum over all such $\tilde{f}$ and $\tilde{g}$ on the right hand side completes the proof.
\end{proof}

Our discrete Pr\'ekopa-Leindler inequality \eqref{our-pl} is obviously the most similar looking form of the classical Prekopa-Leindler inequality \eqref{pl} compared to the other discrete Pr\'ekopa-Leindler type inequalities existing in the literature. Let us also illustrate with an example on how the different Brunn-Minkowski inequalities capture behavior for small $t$: Applying \eqref{KL-prek} when $d=1$ to $f(n) = \mathbbm{1}_A(n) \mu(n)$, $g(n) = \mathbbm{1}_B(n) \mu(n)$, $h(n) = \mathbbm{1}_{\lfloor t A + (1-t) B \rfloor}(n) \mu(n)$, and $k(n) = \mathbbm{1}_{\lceil (1-t) A + t B \rceil}(n) \mu(n)$, where $A, B \subseteq \mathbb{Z}$, $\mu$ is a log-concave probability measure on $\Z$, and
    \begin{align*}
    \lfloor t A + (1-t) B \rfloor &= \{ z \in \mathbb{Z} : z = \lfloor t a + (1-t) b \rfloor, a \in A, b \in B \},
        \\
     \lceil (1-t) A + t B \rceil &= \{ z \in \mathbb{Z} : z = \lceil (1-t) a + t b \rceil, a \in A, b \in B \},
    \end{align*}
    one deduces that
    \begin{align}\label{KL-cor}
       \mu(A) \mu(B) \leq \mu(\lfloor t A + (1-t) B \rfloor) \mu( \lceil (1-t) A+ tB \rceil).
    \end{align}
    Consider $A = \llbracket 0, \infty \llbracket$, $B = \{0\}$, and $\mu(n) = t (1-t)^n$ for $n \geq 0$, inequality \eqref{KL-cor} yields an inequality growing more trivial as $t$ tends to zero,
\begin{align*}
    t \leq 1.
\end{align*}

Also, the same example with $K = A$ and $L = B$, and $f(z) = g(z) = h(z) = t (1-t)^z$ for $z \geq0$ in \eqref{other-pl}, so that $M \cap \Z = A$, and $h^{\circ}(z) = t(1-t)^{z-1}$ for $z \geq 1$ and $h^{\circ}(0) = t$ , yields
$$ t^t \leq 1 + t, $$
which is a non-optimal inequality as $t$ tends to 0, and is growing more trivial as $t$ tends to 1.

Meanwhile, Theorem \ref{thm: main} yields $\mu^{1-t}(A) \mu^t(B) \leq \mu( \mathcal{M}_t(A,B))$ where $\mathcal{M}_t(A,B) \coloneqq \{ z \in \mathbb{Z} : | z - (1-t) a + t b | < 1, a \in A, b \in B\}$, which in this case is $\llbracket 0 , \infty \llbracket$, and reduces to a stronger inequality, which is moreover sharp for $t$ close to zero or $t$ close to 1:
\begin{align*}
    t^t \leq 1.
\end{align*}

\subsection{Dilation Inequalities}

The aim of this section is to establish a discrete analogue of a result of Nazarov-Sodin-Volberg about dilation inequalities \cite{NSV02} (see also \cite{BN07}, \cite{Fra09}). As an application, we will obtain large and small deviations inequalities for discrete log-concave random variables.

Recall the following definition in the continuous setting (see \cite{BN07}): For a given measurable set $A$ of a convex set $K \subset \R^d$ and $\delta \in (0,1)$,
\begin{align}\label{cont-ad}
\overline{A_{\delta}} = \{z \in A : |A \cap \Delta|_1 \geq (1 - \delta) |\Delta|_1, \, \forall \mbox{ intervals } \Delta \subset K \mbox{ such that } z \in \Delta \}.
\end{align}



We adapt this definition to the discrete setting. On $\mathbb{Z}$, we consider a set $\Delta$ to be an interval when $z_1 \leq z_2 \leq z_3$ in $\mathbb{Z}$ with $z_1, z_3 \in \Delta$ implies $z_2 \in \Delta$.  For an interval $\Delta$ with a point $z \in \Delta$ we write $\Delta_z = \Delta \setminus \{z\}$.

\begin{defn} \label{defn: delta interior}
    For $A \subset \mathbb{Z}$ contained in an interval $K \subset \mathbb{Z}$, and $\delta \in (0,1)$, define
    \begin{align*}
        A_\delta = \{ z \in A : | A \cap \Delta_z | \geq (1-\delta) |\Delta_z|, \, \forall \mathrm{\, intervals \,\,} \Delta \subset K \mathrm{\, such \, that \,} z \in \Delta\},
    \end{align*}
    where $|\cdot|$ denotes cardinality.
\end{defn}

For $x,y \in \mathbb{Z}$ we denote by $\Delta(x,y) = \, \rrbracket x ,y \rrbracket \cup \llbracket y , x \llbracket$, the interval between $x$ and $y$, with $x$ removed.
Let us note that in Definition \ref{defn: delta interior}, it suffices to check intervals $\Delta$ of the form $\Delta(z,y)$.  Indeed if $z \in \Delta$ is not an end point, then 
there exist $x$ and $y$ such that $\Delta_z = \Delta(z,x) \cup \Delta(z,y)$ and hence using the result for the restricted class, gives
\begin{align*}
    | A \cap \Delta_z | 
        &=
            | A \cap ( \Delta(z,x) \cup \Delta(z,y) )|
                \\
        &=
            |A \cap \Delta(z,x)| + |A \cap \Delta(z,y)|
                \\
        &\geq
            (1-\delta) ( |\Delta(z,x)| + |\Delta(z,y)| )
                \\
        &=
            (1-\delta) |\Delta_z|.
\end{align*}

As an illustration of Definition \ref{defn: delta interior}, fix $K = \llbracket 0, +\infty \llbracket$ and $\delta \in (0,1)$, and consider $A = \llbracket a, +\infty \llbracket$, for some $a \geq 1$. Let $z \in A$. As noted above, one may only consider intervals of the form $\Delta(z,y)$. 
Note that if $y \geq a$, then $|A \cap \Delta(z,y)| = |\Delta(z,y)|$, and if $y < a$, then $\frac{|A \cap \Delta(z,y)|}{|\Delta(z,y)|} = \frac{z-a}{z-y}$, which is increasing in $y$. Therefore, one may take $y=0$. In this case, $\frac{|A \cap \Delta(z,y)|}{|\Delta(z,y)|} = \frac{z-a}{z}$, which is increasing in $z$, and $\frac{z-a}{z} \geq 1 - \delta$ if and only if $z \geq \frac{a}{\delta}$. We conclude that $A_{\delta} = \llbracket \lceil \frac{a}{\delta} \rceil, +\infty \llbracket$.

If we fix a compact interval $K \subset \mathbb{Z}$ and consider all log-concave probability sequences supported on $K$, to prove $\mu(A) \geq \mu^{\delta}(A_\delta)$, it suffices by Theorem \ref{thm: four function theorem} applied to $f_1 = 1$, $f_2 = \mathbbm{1}_A$, $f_3 = \mathbbm{1}_{A_\delta}$, and $f_4 = 1$ with $\alpha = \delta$ and $\beta = 1$, to prove the result for log-affine random variables supported on $K$. Note that $A_\delta$ is implicitly dependent on the choice of $K$.  Let $A_\delta(K)$ denote $A_\delta$ defined in Definition \ref{defn: delta interior} with the interval $K$. Notice that if $K \subset K'$ then $A_\delta(K') \subset A_\delta(K)$.  Thus to prove that
\begin{align*}
    \mu(A) \geq \mu^{\delta}(A_\delta)
\end{align*}
holds for all log-concave probability measures $\mu$ with support contained in an interval $K$, it suffices to prove the result for log-affine probability measures $\mu$ such that the support of $\mu$ is exactly $K$.

\begin{thm} \label{thm: geometric dilation inequailty}
    Let $\mu$ be a discrete log-concave probability measure, let $\delta \in (0,1)$, and let $A \subset K$, where $K$ is a (possibly infinite) interval, and $A_\delta$ taken with respect to $K$. Then,
    \begin{align} \label{eq: main inequality}
        \mu(A) \geq \mu^\delta(A_\delta) \mu^{1-\delta}(K).
    \end{align}
    
\end{thm}

Inequality \eqref{eq: main inequality} has been established for continuous log-concave measures (with the appropriate definition \eqref{cont-ad}) in \cite{NSV02} (see also \cite{BN07}, \cite{Fra09}).

For the proof of Theorem \ref{thm: geometric dilation inequailty}, we will need an auxiliary function $\Psi(x) = (1-x)^\delta - (1-x)$.  Observe that $\Psi$ is concave on $(0,1)$ and non-negative since $\Psi(0) = \Psi(1) = 0$.  Further $\frac{\Psi(x)}{x}$ is non-increasing, hence with $x_1, x_2, x_1 + x_2 \in [0,1]$ and with $x_1 \leq x_2 \leq x_1 + x_2$,
$
    \frac{\Psi(x_1 + x_2)}{x_1 + x_2} \leq \frac{\Psi(x_2)}{x_2} \leq \frac{\Psi(x_1)}{x_1}
$ 
so that $\Psi(x_2 + x_1) \leq \Psi(x_2) + \frac{x_1 \Psi(x_2)}{x_2} \leq \Psi(x_2) + \Psi(x_1)$.  Inductively, for $x_i \in (0,1)$ with $\sum_{i=1}^n x_i \leq 1$,
\begin{align} \label{eq: sub-additivity}
   \Psi \left( \sum_{i=1}^n x_i \right) \leq  \sum_{i=1}^n \Psi(x_i).
\end{align}

\begin{proof}
    Note that by approximation it suffices to consider the case that $\mu$ is supported on a compact set.  Further by restricting $\mu$ to the set $K$, by $\mu\big|_K(B) = \mu(B \cap K)/ \mu(K)$, it suffices to assume that $\mu(K) =1$, and to prove 
    \begin{align}\label{main ineq}
        \mu(A) \geq \mu^\delta(A_\delta).
    \end{align}
    By Theorem \ref{thm: four function theorem}, it suffices to prove the result when $\mu$ is log-affine.  Assume that $A_\delta^c = \cup_{i=0}^n I_i$ where $I_i$ are disjoint intervals separated by at least one point.  For concreteness, assume $\max I_i \leq \min I_{i+1} -2$, and that $\mu$ is supported on $\llbracket 0, m \rrbracket$ with $ \mu (\{ k \}) = \frac{ 1 - p}{1 - p^{m+1 }}  p^k$ with $p \in (0,1)$.   It suffices to prove 
    \begin{align} \label{eq: sufficient inequality}
        \mu( A \cap I_i) \geq \Psi( \mu(I_i))
    \end{align}
    for $0 \leq i \leq n$.  Indeed, subtracting by $\mu(A_\delta)$, \eqref{main ineq} is equivalent to
    \begin{align*}
        \mu(A) - \mu(A_\delta) 
            &\geq \mu^\delta(A_\delta) - \mu(A_\delta) 
                \\
            &= \Psi(\mu(A_\delta^c))
                \\
            &=
                \Psi \left( \sum_{i=0}^n \mu(I_i) \right).
    \end{align*}
    Applying \eqref{eq: sub-additivity}, and assuming $\mu(A \cap I_i) \geq \Psi(\mu(I_i))$ holds for all $i$,
    \begin{align*}
        \mu(A) - \mu(A_\delta)
            &=
                \mu(A \cap A_\delta^c)
                    \\
            &=
                \sum_{i=0}^n \mu( A \cap I_i)
                    \\
            &\geq
                \sum_{i=0}^n \Psi( \mu(I_i))
                    \\
            &\geq
                \Psi \left( \sum_{i=0}^n \mu(I_i) \right)
                    \\
            &=
                \Psi( \mu(A_\delta^c) ).
    \end{align*}
    To prove \eqref{eq: sufficient inequality}, first consider $I_i = \llbracket a, b\rrbracket$ with $a >0$.  In this case, $a - 1 \in A_\delta$ so that by the definition of $A_\delta$, for all $x \geq a$, $|A  \cap \llbracket a, x \rrbracket | \geq (1-\delta) |\llbracket a, x \rrbracket |.$
    Recall the summation by parts formula,
    \begin{align*}
        \sum_{k = 0}^N f_k g_k = f_N \sum_{k=0}^N g_k + \sum_{j = 0}^{N-1} (f_{j} - f_{j+1}) \sum_{k = 0}^j g_k,
    \end{align*}
    which we apply with $f_k = \mu (\{a + k \}) $, $g_k = \mathbbm{1}_{A}(a + k)$, and $N = b-a$,
    \begin{align*}
        \mu( A \cap I_i) 
            &=
                \sum_{k=0}^{b-a} \mu (\{ a+k \}) \mathbbm{1}_A(a + k)
                    \\
            &=
                \mu (\{b \})\sum_{k=0}^{b-a} \mathbbm{1}_A(a + k) + \sum_{j = 0}^{b-a-1} (\mu (\{a+j \}) - \mu (\{a+j+1\})) \sum_{k = 0}^j \mathbbm{1}_{A}( a + k)
                    \\
            &=
              \mu (\{b \})| A \cap \llbracket a , b \rrbracket | + \sum_{j=0}^{b-a-1} (\mu (\{a+j \}) - \mu (\{a+j+1\})) | A \cap \llbracket a , a+j \rrbracket |.
    \end{align*}
    By $|A  \cap \llbracket a, x \rrbracket | \geq (1-\delta) |\llbracket a, x \rrbracket |$, $\mu (\{a+j \}) - \mu (\{a+j+1\}) \geq 0$, and an application of summation by parts again with the constant function $1$ replacing $g_k = \mathbbm{1}_{A}(a + k )$,
    \begin{align*}
        \mu (\{b \}) | A \cap \llbracket a , b \rrbracket | &+ \sum_{j=0}^{b-a-1} (\mu (\{a+j \}) - \mu (\{a+j+1\}))| A \cap \llbracket a , a+j \rrbracket | \\
            &\geq
            (1-\delta) \left( \mu (\{b \}) | \llbracket a , b \rrbracket | + \sum_{j=0}^{b-a-1} (\mu (\{a+j \}) - \mu (\{a+j+1\})) | \llbracket a , a+j \rrbracket | \right)
                \\
            &=
            (1-\delta) \mu(\llbracket a, b \rrbracket).
    \end{align*}
    Thus $ \mu( A \cap I_i ) \geq (1-\delta) \mu(I_i) \geq \Psi(\mu( I_i))$, where the second inequality follows from the arithmetic-geometric means inequality,
    $$ \Psi(\mu( I_i)) = ( 1- \mu(I_i) )^\delta  - (1 - \mu(I_i) \leq \delta ( 1 - \mu(I_i)) + (1-\delta) - ( 1 - \mu(I_i)) = (1-\delta) \mu(I_i). $$

    Now suppose $I_i = \llbracket 0 , b -1 \rrbracket$. Then the inequality we pursue is
    \begin{align*}
        \mu( A \cap I_i) \geq \mu^\delta(\llbracket b , m \rrbracket) - \mu(\llbracket b , m \rrbracket).
    \end{align*}
    However, since $|A \cap I_i | \geq (1-\delta) |I_i| = (1-\delta) b$, and $\mu$ is decreasing,
    we have $\mu(A \cap I_i) \geq \mu( \llbracket \floor*{\delta b}, b-1 \rrbracket)$. Rearranging, it suffices to prove
    \begin{align*}
        \mu(\llbracket \floor*{ \delta b }, m \rrbracket) \geq \mu^\delta( \llbracket b, m \rrbracket).
    \end{align*}
    However the above is equivalent to 
    \begin{align*}
        \mu(\llbracket \floor*{ (1-\delta) 0 + \delta b }, m \rrbracket) \geq \mu^\delta( \llbracket b, m \rrbracket) \mu^{1-\delta}(\llbracket 0 ,m \rrbracket ),
    \end{align*}
    which follows from Theorem \ref{thm: main} applied to indicator functions of intervals.
\end{proof}

\subsection{Large and Small Deviations Inequalities}

In this section, we develop large and small deviations inequalities for discrete log-concave random variables.

Recall the following definition in the continuous setting (see \cite{BN07}): For a given measurable function $f$ on $\R^d$ and $\varepsilon \in (0,1)$,
$$ \overline{\delta_f(\varepsilon)} = \sup_{x,y \in \R^d} |\{t \in (0,1) : |f((1-t)x + ty)| \leq \varepsilon |f(x)| \}|_1. $$

We adapt this definition to the discrete setting.

\begin{defn}\label{modus}
    For a (possibly infinite) interval $K \subset \mathbb{Z}$ and $\varepsilon \in (0,1)$, define the modulus of regularity $\delta_f(\varepsilon)$ to a function $f \colon K \to \mathbb{R} $ by
    \begin{align*}
        \delta_f(\varepsilon) = \sup_{\substack{x \neq y \\ x,y \, \in \, K}} \frac{ |\{ z \in \Delta(x,y) : |f(z)| \leq \varepsilon |f(x)| \} |}{|\Delta(x,y)|}
    \end{align*}
\end{defn}

The next proposition illustrates Definition \ref{modus}.

\begin{prop} \label{prop: f(x) = x modulus}
    For $ t > 1$, $f(x) = x$, and $K = \llbracket m, \infty \llbracket$ for $m \geq 1$,
    \begin{align*}
        \delta_f(1/t) \leq  \Psi_f(t) \coloneqq \begin{cases}
                                \frac 1 t &\hbox{ for } t \geq \frac{m}{m-1}
                                    \\
                                \frac{1}{(t-1)m} &\mbox{ for } t \in [ \frac{m+1}{m}, \frac{m}{m-1}]
                                    \\
                                1 &\mbox{ for } t \in (1,\frac{m+1}{m}]
        \end{cases}.
    \end{align*}
    In particular, $\delta_f(1/t) \leq \frac{2}{t}$.
\end{prop}

\begin{proof}
Since $\delta_f \leq 1$ by definition, the inequality is trivial for $t \leq \frac{m+1}{m}$ and we can assume that $t > \frac{m+1}{m}$.  Note that
$$ \left\{ z \in \Delta(x,y) : |f(z)| \leq \varepsilon |f(x)| \right\} = \left\{ z \in \Delta(x,y) : z \leq \frac{x}{t} \right\}, $$
is empty if $x \leq y$ or $x < t m$. Assume thereafter that $y<x$ and $tm \leq x$. In this case,
$$ \left\{ z \in \Delta(x,y) : z \leq \frac{x}{t} \right\} = \left\llbracket y , \left\lfloor \frac{x}{t} \right\rfloor \right\rrbracket. $$
Hence,
$$ \delta_f(1/t) = \sup_{x,y} \frac{\lfloor \frac{x}{t} \rfloor - y + 1}{x-y} = \sup_{ tm \leq x} \frac{\lfloor \frac{x}{t} \rfloor - m + 1}{x-m}  \leq \sup_{tm \leq x} \frac{ \frac{x}{t}  - m + 1}{x-m} $$
as $\frac{\lfloor \frac{x}{t} \rfloor - y + 1}{x-y}$ is decreasing in $y$ since $\lfloor \frac{x}{t} \rfloor  + 1 \leq x$.

Denote $u(x) = \frac{ \frac{x}{t} - m + 1}{x-m}$
, and observe that $u'(x) = \frac{t(m-1) -m}{t(x-m)^2} \geq 0$ when $t \geq \frac{m}{m-1}$ and hence $\sup u = \lim_{x \to \infty} u(x) = \frac 1 t$ in this case. Meanwhile $ t \leq \frac{m}{m-1}$ implies $u'(x) \leq 0$ so that $\sup u = u( tm) = \frac{1}{(t-1)m}$. This gives $\delta_f(1/t) \leq \Psi_f(t)$.
\end{proof}

\begin{thm} \label{thm: functional dilation theorem}
    Let $\mu$ be a discrete log-concave probability measure supported on $K \subset \Z$. For all $\varepsilon \in (0,1), \lambda >0$, and $f \colon K \to \R$ with modulus of regularity $\delta = \delta_f(\varepsilon)$, we have 
    \begin{align*}
        \mu (\{ |f| > \lambda \varepsilon \}) \geq \mu^\delta (\{ |f| \geq \lambda \}).
    \end{align*}
\end{thm}

\begin{proof}
    Define $A = \{ w : |f(w)| > \lambda \varepsilon \}$, and consider $x$ such that $f(x) \geq \lambda$. By the definition of the modulus of regularity, for any $y \in K$,
    \begin{align} \label{eq: equation to be rearranged}
       | \{ z \in \Delta(x,y) : |f(z)| \leq \varepsilon \lambda \} |
        \leq
        | \{ z \in \Delta(x,y) : |f(z)| \leq \varepsilon |f(x)| \} | 
        \leq \delta | \Delta(x,y) |.
    \end{align}
Since $| \{ z \in \Delta(x,y) : |f(z)| \leq \varepsilon \lambda \} | + | \{ z \in \Delta(x,y) : |f(z)| > \varepsilon \lambda \} |  = |\Delta(x,y)|$, rearranging \eqref{eq: equation to be rearranged} gives
\begin{align*}
    | \{ z \in \Delta(x,y) : |f(z)| > \varepsilon \lambda \} | \geq (1-\delta) |\Delta(x,y)|.
\end{align*}
Therefore $\{ x : |f(x)| \geq \lambda \} \subseteq A_\delta$. Thus applying Theorem \ref{thm: geometric dilation inequailty} we have
\begin{align*}
    \mu (\{ |f| > \lambda \varepsilon \})
        = \mu(A)
        \geq
            \mu^\delta(A_\delta)
        \geq
            \mu^\delta (\{ |f| \geq \lambda \}),
\end{align*}
which is our conclusion.
\end{proof}

We note that Theorem \ref{thm: functional dilation theorem} implies Theorem \ref{thm: geometric dilation inequailty} as can be seen by taking $f$ taking no more than three values, so the two theorems are in fact equivalent.

For $q \in (0,1)$, we consider $m_q$ to be a $q$-quantile of $|f|$ with respect to a measure $\mu$ when $\mu (\{ |f| > m_q \}) \leq q \leq \mu ( \{ |f| \geq m_q \})$.

\begin{cor}\label{large deviation}
    For a log-concave probability measure $\mu$ and $|f|$ with $q$-quantile $m_q$,
    \begin{align*}
        \mu \{ |f| \geq m_q t \} \leq q^{1/{\delta_f(1/t)}}
    \end{align*}
    holds for all $t > 1$.
\end{cor}

\begin{proof}
Taking $t > 1$, and applying Theorem \ref{thm: functional dilation theorem} with $\lambda =m_q t$ and $\varepsilon = \frac 1 t$, we have
\begin{align*}
    q \geq \mu (\{ |f| > m_q \})
        =
            \mu \{ |f| > \lambda \varepsilon \}
        \geq
            \mu^{\delta_f(\varepsilon)} (\{ |f| \geq \lambda \})
        =
            \mu^{\delta_f(1/t)} (\{ |f| \geq m_q t \}).
\end{align*}
\end{proof}

\begin{cor}\label{small deviation}
    For a log-concave probability measure $\mu$ and $|f|$ with $q$-quantile $m_q$,
    \begin{align*}
        \mu (\{ |f| \leq m_q \varepsilon \}) \leq 1 - q^{\delta_f(\varepsilon)} \leq \delta_f(\varepsilon) \log(1/q),
    \end{align*}
    holds for all $\varepsilon \in (0,1)$.
\end{cor}

\begin{proof}
    Applying Theorem \ref{thm: functional dilation theorem} with $\lambda = m_q$, gives 
    \begin{align*}
        \mu (\{ |f| > m_q \varepsilon \})
            \geq
                \mu^{\delta_f(\varepsilon)} (\{ |f| \geq m_q \})
            \geq
                q^{\delta_f(\varepsilon)}.
    \end{align*}
    The second inequality is a consequences of $1 - e^{-y} \leq y$ applied to $y = \delta_f(\varepsilon) \log \frac 1 q$.
\end{proof}


Corollaries \ref{large deviation} and \ref{small deviation}, with $f(x) = x$ and Proposition \ref{prop: f(x) = x modulus} can be summarized in probabilistic notation as the following large and small deviation inequalities. In what follows $X$ is a random variable with a log-concave distribution, and with $f(x) = x$, so that $m_q$ the $q$-th quantiles of $f$, are characterized by satisfying $\mathbb{P}(X > m_q) \leq q \leq \mathbb{P}(X \geq m_q)$. 

\begin{cor}\label{deviations}

Let $X$ be a discrete log-concave random variable supported on $\N \setminus \{0\}$. Then, for all $t > 1$ and $\varepsilon \in (0,1)$,
$$ \P(X \geq m_q \, t) \leq  q^{\frac t 2}, \qquad \P(X \leq m_q \, \varepsilon) \leq  1- q^{2 \varepsilon} \leq \ 2\log(1/q) \varepsilon. 
$$
\end{cor}

Let us observe that the inequality is sharp in the following case. Let $q \in (0,1)$ and consider $X$ such that $1 -\mathbb{P}(X = 1) = q = \mathbb{P}(X = 2)$. One can take $m_q = 1$ and thus for $t \in (1,2)$, 
\[
q^{\frac t 2} -\mathbb{P}(Y \geq m_q t) = q^{\frac t 2}- q. \]
Taking $t \to 2$ shows the inequality to be sharp.
 For the small deviation inequality, one can consider the same random variable $X$ as above, so that one can take $m_q = 2$, and then for $\varepsilon \in (1/2,1)$, 
\[
    1 - q^{2 \varepsilon} - \mathbb{P}(Y \leq m_q \varepsilon) = q - q^{2\varepsilon}.
\]
Taking $\varepsilon \to 1/2$ shows the inequality to be sharp. We note that the constant $2$ may not be optimal for a linear bound in the second inequality, but cannot be improved beyond a factor of $2$ as can be checked with a geometric distribution.

One can deduce large deviation inequalities for discrete log-concave random variable supported on $\N$.

\begin{cor}

Let $X$ be a discrete log-concave random variable supported on $\N$. Then, for all $u > m_q$,
\begin{eqnarray} \label{eq: large deviation with 0 in there}
\P(X \geq u) \leq e^{-(u+1)\frac{\log(1/q)}{2(1 + m_q)}}.
\end{eqnarray}
In particular, taking $q = \frac{1}{2}$, we have that for all $u > \mathrm{Med}(X)$,
\begin{eqnarray*}
\P(X \geq u) \leq e^{-(u+1)\frac{\log(2)}{2(1+\mathrm{Med}(X))}},
\end{eqnarray*}
where $Med(X)$ denotes a median of $X$.

\end{cor}

\begin{proof}
Define $Y = X+1$ so that $Y$ is discrete log-concave on $\N \setminus \{0\}$. Note that $m_q(Y)$, the $q$-quantile of $Y$, satisfies $m_q(Y) = m_q + 1$. Then, by Corollary \ref{deviations}, for all $s > m_q + 1$,
\begin{eqnarray*}
\P(X \geq s - 1) = \P(Y \geq s) & \leq & e^{-s\frac{\log(1/q)}{2 m_q(Y)}},
\end{eqnarray*}
and the result follows.
\end{proof}

Observe that equality holds in \eqref{eq: large deviation with 0 in there} when $X$ is a random variable with a Bernoulli distribution of parameter $q$, $m_q = 0$, and $u = 1$.

\vskip5mm
\noindent
{\bf Acknowledgements.}

The authors would like to thank the referees whose valuable comments greatly improved the presentation of the article.

\vskip5mm
\noindent
{\bf Data availability.}

Data sharing not applicable to this article as no datasets were generated or analysed during the current study.
\vskip1cm

\bibliographystyle{plain}
\bibliography{bibibi}

\begin{thebibliography}{10}

\bibitem{AHZ17}
M.~Alexander, M.~Henk, and A.~Zvavitch.
\newblock A discrete version of {K}oldobsky's slicing inequality.
\newblock {\em Israel J. Math.}, 222(1):261--278, 2017.

\bibitem{amelunxen2014living}
D.~Amelunxen, M.~Lotz, M.B. McCoy, and J.A. Tropp.
\newblock Living on the edge: Phase transitions in convex programs with random
  data.
\newblock {\em Information and Inference: A Journal of the IMA}, 3(3):224--294,
  2014.

\bibitem{Bob07:gafa}
S.~G. Bobkov.
\newblock On isoperimetric constants for log-concave probability distributions.
\newblock In {\em Geometric aspects of functional analysis}, volume 1910 of
  {\em Lecture Notes in Math.}, pages 81--88. Springer, Berlin, 2007.

\bibitem{BM11:aop}
S.~G. Bobkov and M.~Madiman.
\newblock Concentration of the information in data with log-concave
  distributions.
\newblock {\em Ann. Probab.}, 39(4):1528--1543, 2011.

\bibitem{Sergey2015localization}
S.~G. Bobkov and J.~Melbourne.
\newblock Localization for infinite-dimensional hyperbolic measures.
\newblock {\em Doklady Mathematics}, 91(3):297--299, 2015.

\bibitem{Sergey2016ProbabilitySurveys}
S.~G. Bobkov and J.~Melbourne.
\newblock Hyperbolic measures on infinite dimensional spaces.
\newblock {\em Probability Surveys}, 13:57--88, 2016.

\bibitem{BN07}
S.~G. Bobkov and F.~L. Nazarov.
\newblock Sharp dilation-type inequalities with a fixed parameter of convexity
  ({R}ussian).
\newblock {\em Zap. Nauch. Sem. P.O.M.I.}, 351:54--78, 2007.

\bibitem{bobkov2008sharp}
S.~G. Bobkov and F.~L. Nazarov.
\newblock Sharp dilation-type inequalities with a fixed parameter of convexity.
\newblock {\em Journal of Mathematical Sciences}, 152(6):826--839, 2008.

\bibitem{borcea2009negative}
J.~Borcea, P.~Br{\"a}nd{\'e}n, and T.~Liggett.
\newblock Negative dependence and the geometry of polynomials.
\newblock {\em Journal of the American Mathematical Society}, 22(2):521--567,
  2009.

\bibitem{Bre94}
F.~Brenti.
\newblock Log-concave and unimodal sequences in algebra, combinatorics, and
  geometry: an update.
\newblock In {\em Jerusalem combinatorics '93}, volume 178 of {\em Contemp.
  Math.}, pages 71--89. Amer. Math. Soc., Providence, RI, 1994.

\bibitem{Br15}
P.~Brändén.
\newblock Unimodality, log-concavity, real-rootedness and beyond.
\newblock {\em Handbook of enumerative combinatorics, Discrete Math. Appl.},
  CRC Press:437--483, 2015.

\bibitem{BH20}
P.~Brändén and J.~Huh.
\newblock Lorentzian polynomials.
\newblock {\em Ann. of Math.}, 192(2):821--891, 2020.

\bibitem{CV18}
B.~Cousins and S.~Vempala.
\newblock Gaussian cooling and $o^*(n^3)$ algorithms for volume and gaussian
  volume.
\newblock {\em SIAM J. Comput.}, 47(3):1237–1273, 2018.

\bibitem{Dyer-Frieze}
M.~E. Dyer and A.~M. Frieze.
\newblock Computing the volume of a convex body: a case where randomness
  provably helps.
\newblock {\em Proc. of AMS Symposiumon Probabilistic Combinatorics and Its
  Applications}, pages 123--170, 1991.

\bibitem{Eld13}
R.~Eldan.
\newblock Thin shell implies spectral gap up to polylog via a stochastic
  localization scheme.
\newblock {\em Geom. Funct. Anal.}, 23(2):532--569, 2013.

\bibitem{fekete1912problem}
M.~Fekete.
\newblock {\"U}ber ein problem von {L}aguerre.
\newblock {\em Rendiconti del Circolo Matematico di Palermo (1884-1940)},
  34(1):89--120, 1912.

\bibitem{Fra09}
M.~Fradelizi.
\newblock Concentration inequalities for {$s$}-concave measures of dilations of
  {B}orel sets and applications.
\newblock {\em Electron. J. Probab.}, 14:no. 71, 2068--2090, 2009.

\bibitem{FG04}
M.~Fradelizi and O.~Gu{\'e}don.
\newblock The extreme points of subsets of {$s$}-concave probabilities and a
  geometric localization theorem.
\newblock {\em Discrete Comput. Geom.}, 31(2):327--335, 2004.

\bibitem{FG06}
M.~Fradelizi and O.~Gu{\'e}don.
\newblock A generalized localization theorem and geometric inequalities for
  convex bodies.
\newblock {\em Adv. Math.}, 204(2):509--529, 2006.

\bibitem{Gar02}
R.~J. Gardner.
\newblock The {B}runn-{M}inkowski inequality.
\newblock {\em Bull. Amer. Math. Soc. (N.S.)}, 39(3):355--405 (electronic),
  2002.

\bibitem{GG01}
R.~J. Gardner and P.~Gronchi.
\newblock A {B}runn-{M}inkowski inequality for the integer lattice.
\newblock {\em Trans. Amer. Math. Soc.}, 353(10):3995--4024 (electronic), 2001.

\bibitem{goldstein2017gaussian}
L.~Goldstein, I.~Nourdin, and G.~Peccati.
\newblock Gaussian phase transitions and conic intrinsic volumes: Steining the
  {S}teiner formula.
\newblock {\em The Annals of Applied Probability}, 27(1):1--47, 2017.

\bibitem{GRST14}
N.~Gozlan, C.~Roberto, P.-M. Samson, and P.~Tetali.
\newblock Displacement convexity of entropy and related inequalities on graphs.
\newblock {\em Probab. Theory Related Fields}, 160(1-2):47--94, 2014.

\bibitem{GRST19}
N.~Gozlan, C.~Roberto, P.-M. Samson, and P.~Tetali.
\newblock {Transport proofs of some discrete variants of the Pr\'ekopa-Leindler
  inequality}.
\newblock {\em Ann. Sc. Norm. Super. Pisa Cl. Sci.}, (5) 22(3):1207--1232,
  2021.

\bibitem{Gue99}
O.~Gu{\'e}don.
\newblock Kahane-{K}hinchine type inequalities for negative exponent.
\newblock {\em Mathematika}, 46(1):165--173, 1999.

\bibitem{gurvits2009multivariate}
L.~Gurvits.
\newblock On multivariate {N}ewton-like inequalities.
\newblock In {\em Advances in combinatorial mathematics}, pages 61--78.
  Springer, 2009.

\bibitem{HKS}
D.~Halikias, B.~Klartag, and B.~Slomka.
\newblock Discrete variants of {B}runn-{M}inkowski type inequalities.
\newblock {\em Ann. Fac. Sci. Toulouse Math.}, (6) 30(2):267--279, 2021.

\bibitem{havrilla2021khinchin}
A.~Havrilla, P.~Nayar, and T.~Tkocz.
\newblock Khinchin-type inequalities via {H}adamard's factorisation.
\newblock {\em International Mathematics Research Notices,
  https://doi.org/10.1093/imrn/rnab313}, 2021.

\bibitem{hoggar1974chromatic}
S.~G. Hoggar.
\newblock Chromatic polynomials and logarithmic concavity.
\newblock {\em Journal of Combinatorial Theory, Series B}, 16(3):248--254,
  1974.

\bibitem{HMMS22}
J.~Huh, J.~P. Matherne, K.~M\'esz\'aros, and A.~St.~Dizier.
\newblock Logarithmic concavity of schur and related polynomials.
\newblock {\em Trans. Amer. Math. Soc.}, 375(6):4411--4427, 2022.

\bibitem{LNZ20}
D.~Iglesias, J.~Yepes Nicol\'as, and A.~Zvavitch.
\newblock Brunn-{M}inkowski inequalities for the lattice point enumerator.
\newblock {\em Adv. Math.}, 370:25pp, 2020.

\bibitem{joag1983negative}
K.~Joag-Dev and F.~Proschan.
\newblock Negative association of random variables with applications.
\newblock {\em The Annals of Statistics}, pages 286--295, 1983.

\bibitem{Joh07}
O.~Johnson.
\newblock Log-concavity and the maximum entropy property of the {P}oisson
  distribution.
\newblock {\em Stochastic Process. Appl.}, 117(6):791--802, 2007.

\bibitem{JKM13}
O.~Johnson, I.~Kontoyiannis, and M.~Madiman.
\newblock Log-concavity, ultra-log-concavity, and a maximum entropy property of
  discrete compound {Poisson} measures.
\newblock {\em Discrete Appl. Math.}, 161:1232--1250, 2013.
\newblock DOI: 10.1016/j.dam.2011.08.025.

\bibitem{KLS95}
R.~Kannan, L.~Lov{\'a}sz, and M.~Simonovits.
\newblock Isoperimetric problems for convex bodies and a localization lemma.
\newblock {\em Discrete Comput. Geom.}, 13(3-4):541--559, 1995.

\bibitem{KLS97}
R.~Kannan, L.~Lov{\'a}sz, and M.~Simonovits.
\newblock Random walks and an {$O\sp *(n\sp 5)$} volume algorithm for convex
  bodies.
\newblock {\em Random Structures Algorithms}, 11(1):1--50, 1997.

\bibitem{KG71}
J.~Keilson and H.~Gerber.
\newblock Some results for discrete unimodality.
\newblock {\em J. Amer. Statist. Assoc.}, 66(334):386--389, 1971.

\bibitem{Klar17}
B.~Klartag.
\newblock Needle decompositions in {R}iemannian geometry.
\newblock {\em Mem. Amer. Math. Soc.}, 249(1180), 2017.

\bibitem{KL19}
B.~Klartag and J.~Lehec.
\newblock Poisson processes and a log-concave {B}ernstein theorem.
\newblock {\em Stud. Math.}, 247(1):85–107, 2019.

\bibitem{LV17:1}
Y.~T. Lee and S.~Vempala.
\newblock Eldan's stochastic localization and the {KLS} hyperplane conjecture:
  An improved lower bound for expansion.
\newblock {\em 58th Annual IEEE Symposium on Foundations of Computer
  Science--FOCS 2017}, IEEE Computer Soc., Los Alamitos, CA:998--1007, 2017.

\bibitem{LV17:2}
Y.~T. Lee and S.~Vempala.
\newblock The {KLS} conjecture.
\newblock {\em Current Developments in Mathematics 2017}, 1–36, Int. Press,
  Somerville, MA, 2019.

\bibitem{Lei72a}
L.~Leindler.
\newblock On a certain converse of {H}\"older's inequality.
\newblock In {\em Linear operators and approximation ({P}roc. {C}onf.,
  {O}berwolfach, 1971)}, pages 182--184. Internat. Ser. Numer. Math., Vol. 20.
  Birkh\"auser, Basel, 1972.

\bibitem{LS93}
L.~Lov{\'a}sz and M.~Simonovits.
\newblock Random walks in a convex body and an improved volume algorithm.
\newblock {\em Random Structures Algorithms}, 4(4):359--412, 1993.

\bibitem{mccoy2014steiner}
M.~B. McCoy and J.~A. Tropp.
\newblock From {S}teiner formulas for cones to concentration of intrinsic
  volumes.
\newblock {\em Discrete \& Computational Geometry}, 51(4):926--963, 2014.

\bibitem{mcnamara2010infinite}
P.~R.~W. McNamara and B.~E. Sagan.
\newblock Infinite log-concavity: developments and conjectures.
\newblock {\em Advances in Applied Mathematics}, 44(1):1--15, 2010.

\bibitem{MT20:isit}
J.~Melbourne and T.~Tkocz.
\newblock On the {R}\'enyi entropy of log-concave sequences.
\newblock {\em Proc. IEEE Int. Symp. Inform. Theory}, 2020.

\bibitem{Nayar2012khinchine}
P.~Nayar and K.~Oleszkiewicz.
\newblock Khinchine type inequalities with optimal constants via ultra
  log-concavity.
\newblock {\em Positivity}, 16(2):359--371, 2012.

\bibitem{NSV02}
F.~Nazarov, M.~Sodin, and A.~Volberg.
\newblock The geometric {K}annan-{L}ov\'asz-{S}imonovits lemma, dimension-free
  estimates for the distribution of the values of polynomials, and the
  distribution of the zeros of random analytic functions.
\newblock {\em Algebra i Analiz}, 14(2):214--234, 2002.

\bibitem{NSV03}
F.~Nazarov, M.~Sodin, and A.~Volberg.
\newblock Local dimension-free estimates for volumes of sublevel sets of
  analytic functions.
\newblock {\em Israel J. Math.}, 133:269--283, 2003.

\bibitem{OV12}
Y.~Ollivier and C.~Villani.
\newblock A curved {B}runn-{M}inkowski inequality on the discrete hypercube,
  or: what is the {R}icci curvature of the discrete hypercube?
\newblock {\em SIAM J. Discrete Math.}, 26(3):983--996, 2012.

\bibitem{Pem00}
R.~Pemantle.
\newblock Towards a theory of negative dependence.
\newblock {\em J. Math. Phys.}, 41(3):1371--1390, 2000.
\newblock Probabilistic techniques in equilibrium and nonequilibrium
  statistical physics.

\bibitem{pitman1997probabilistic}
J.~Pitman.
\newblock Probabilistic bounds on the coefficients of polynomials with only
  real zeros.
\newblock {\em Journal of Combinatorial Theory, Series A}, 77(2):279--303,
  1997.

\bibitem{Pre71}
A.~Pr{\'e}kopa.
\newblock Logarithmic concave measures with application to stochastic
  programming.
\newblock {\em Acta Sci. Math. (Szeged)}, 32:301--316, 1971.

\bibitem{RYZ17}
D.~Ryabogin, V.~Yaskin, and N.~Zhang.
\newblock Unique determination of convex lattice sets.
\newblock {\em Discrete Comput. Geom.}, 57(3):582–589, 2017.

\bibitem{saumard2014log}
A.~Saumard and J.~A. Wellner.
\newblock Log-concavity and strong log-concavity: a review.
\newblock {\em Statistics surveys}, 8:45, 2014.

\bibitem{savage2015}
C.~Savage and M.~Visontai.
\newblock The s-eulerian polynomials have only real roots.
\newblock {\em Transactions of the American Mathematical Society},
  367(2):1441--1466, 2015.

\bibitem{schneider2014convex}
R.~Schneider.
\newblock {\em Convex bodies: the {B}runn--{M}inkowski theory}.
\newblock Number 151 in Encyclopedia of Mathematics and its Applications.
  Cambridge university press, 2014.

\bibitem{schoenberg1955zeros}
I.~J. Schoenberg.
\newblock On the zeros of the generating functions of multiply positive
  sequences and functions.
\newblock {\em Annals of Mathematics}, pages 447--471, 1955.

\bibitem{S20}
B.~Slomka.
\newblock {A Remark on discrete Brunn-Minkowski type inequalities via
  transportation of measure}.
\newblock {\em Preprint, {\tt arXiv:2008.00738}}, 2020.

\bibitem{stanley1981two}
R.~P. Stanley.
\newblock Two combinatorial applications of the {A}leksandrov-{F}enchel
  inequalities.
\newblock {\em Journal of Combinatorial Theory, Series A}, 31(1):56--65, 1981.

\bibitem{Sta89}
R.~P. Stanley.
\newblock Log-concave and unimodal sequences in algebra, combinatorics, and
  geometry.
\newblock In {\em Graph theory and its applications: {E}ast and {W}est
  ({J}inan, 1986)}, volume 576 of {\em Ann. New York Acad. Sci.}, pages
  500--535. New York Acad. Sci., New York, 1989.

\bibitem{stanley2000positivity}
R.~P. Stanley.
\newblock Positivity problems and conjectures in algebraic combinatorics.
\newblock {\em Mathematics: frontiers and perspectives}, 295:319, 2000.

\bibitem{Wal76}
D.~W. Walkup.
\newblock P\'olya sequences, binomial convolution and the union of random sets.
\newblock {\em J. Appl. Probability}, 13(1):76--85, 1976.

\bibitem{wang2008q}
J.~Wang and H.~Zhang.
\newblock q-weighted log-concavity and q-product theorem on the normality of
  posets.
\newblock {\em Advances in Applied Mathematics}, 41(3):395--406, 2008.

\end{thebibliography}

\end{document}